\documentclass[pdflatex,sn-mathphys-ay]{sn-jnl}%
\usepackage{graphicx}
\usepackage{graphics}%
\usepackage{multirow}%
\usepackage{amsmath,amssymb,amsfonts}%
\usepackage{amsthm}%
\usepackage{mathrsfs}%
\usepackage[title]{appendix}%
\usepackage{xcolor}%
\usepackage{textcomp}%
\usepackage{manyfoot}%
\usepackage{booktabs}%
\usepackage{algorithm}%
\usepackage{algorithmicx}%
\usepackage{algpseudocode}%
\usepackage{listings}%
\usepackage{natbib}
\usepackage{bm}
\usepackage{commath}
\usepackage[T1]{fontenc} 
\usepackage{enumitem}
\usepackage{mathtools}
\usepackage{rotating}
\theoremstyle{thmstyleone}%
\newtheorem{theorem}{Theorem}%
\newtheorem{proposition}[theorem]{Proposition}%

\theoremstyle{thmstyletwo}%
\newtheorem{remark}{Remark}%

\theoremstyle{thmstylethree}%
\newtheorem{assumption}{Assumption}

\newtheorem{lemma}{Lemma}

\newcommand{\BB}{\mathbb{B}}
\newcommand{\DD}{\mathbb{D}}
\newcommand{\CC}{\mathbb{C}}
\newcommand{\C}{\mathcal{C}}
\newcommand{\EE}{{\mathbb{E}}}
\newcommand{\GG}{{\mathbb{G}}}

\newcommand{\II}{{\mathbb{I}}}

\newcommand{\NN}{{\mathbb{N}}}
\newcommand{\RR}{{\mathbb{R}}}

\newcommand{\FF}{\mathbb{F}}

\newcommand{\Var}{\mathbb{V} {\rm ar}}

\newcommand{\n}{{\noindent}}

\begin{document}

\title[Article Title]{Testing the equality of the copulas of two multivariate populations with Bernstein polynomials}

\author[1]{\fnm{Guanjie} \sur{Lyu}}\email{lvg@uwindsor.ca}

\author*[1]{\fnm{Mohamed } \sur{Belalia}}\email{Mohamed.Belalia@uwindsor.ca}

\affil[1]{\orgdiv{Department of Mathematics and Statistics}, \orgname{University of Windsor}, \orgaddress{\street{401 Sunset Avenue}, \city{Windsor}, \postcode{N9B 3P4}, \state{ON}, \country{Canada}}}

\abstract{The empirical Bernstein copula process is utilized to conduct tests for the equality of the dependence structures  between two samples. To this end, three tests statistics are proposed, and their asymptotic properties are established under mild regularity conditions. Furthermore, in order to obtain a valid p-value, subsampling and multiplier bootstrap methods are introduced and investigated for the Bernstein based-test statistics. Through a simulation study, it is shown that Bernstein-based tests significantly outperform tests based on the empirical copula.}

\keywords{Empirical Bernstein copula process, empirical copula process, multiplier bootstrap, subsampling}

\maketitle

\section{Introduction}
\label{sec:introduction}

Copulas play a crucial role in characterizing the complete dependence structure between random variables, making them indispensable in various fields such as statistics, finance, and actuarial science. They have gained widespread recognition and are considered a fundamental tool. One notable related topic is the testing of copula equality between different groups, which has received growing interest in recent years due to its practical implications; see~\cite{Remillard2009},~\citet{Bouzebda&ElFaouzi&Zari2011},~\citet{Bouzebda&ElFaouzi2012},~\citet{Seo2021} and~\citet{Bakam2024smooth}. In particular,~\citet{Bakam2024smooth} recently proposed a smooth test for the equality of copulas based on a decomposition using Legendre polynomials. The identification of similarities in dependence structures has significant relevance across diverse domains. For instance,~\citet{Dupuis2009empirical} have demonstrated that the dependence structures between credit default swaps and equity returns can differ significantly. Furthermore,~\citet{Szolgay2016regional} conducted an exploration of the homogeneity of the dependence structures pertaining to flood peaks and their corresponding volumes, both within different types of floods and in various basins.

However, the underlying copula function is often unknown, necessitating the adoption of an appropriate parametric model or, more commonly, a nonparametric approach. Among nonparametric estimators, the empirical copula, developed by~\citet{Ruschendorf1976}, stands as the most popular. A multitude of nonparametric tests, including the equality tests mentioned in the previous paragraph, have been built upon this approach in the literature. One major drawback of the empirical copula estimator is its discreteness, which violates the continuity of the copula function when the marginals are continuous. Consequently, the empirical copula is not a genuine copula, and this limitation becomes more pronounced when the sample size is relatively small. To address this issue, several smooth estimators have been introduced. For example,~\citet{Hall2006} and~\citet{Morettin2010} explored wavelet-smoothed empirical copula, while~\citet{Fermanian2004},~\cite{Chen2007} and~\citet{Omelka2009} investigated kernel-smoothed empirical copula. For a comprehensive overview of the nonparametric estimation of the copula and copula density, we refer to~\citet{Charpentier2007}.

The goal of this paper is to develop and evaluate new methods for testing the equality of copulas between two different groups using the empirical Bernstein copula. The choice of empirical Bernstein copula is motivated by two main factors. First, the empirical Bernstein copula is a genuine copula when the order of the Bernstein polynomial is chosen as the divisor of the sample size~\citep[see][]{Segers2017}. This strategy was used in our simulation study. In addition, estimation based on Bernstein polynomials is known to be asymptotically bias-free at boundary points, as compared to kernel-based methods that often suffer from excess bias near the boundaries~\citep[see][]{Sancetta2004,Jansen2012,Leblanc2012b,Belalia2016,Belalia2017b}. A comprehensive discussion on boundary bias for kernel-based methods can be found in~\citet{Chen2007}. Secondly, the empirical Bernstein copula is a polynomial and therefore possesses all partial derivatives, which are crucial for constructing our two resampling methods. Furthermore, as noted in~\citet{Neumann2019}, ``in practice, it may not be the most important which smoothing method to choose, while it is recommendable to smooth at all.'' The empirical Bernstein copula offers a favorable balance between simplicity and effectiveness compared to other nonparametric estimators.

In the context of hypothesis testing with nonparametric estimators, standard Monte Carlo procedures cannot obtain valid p-values, as elaborated in~\hyperref[sec:multiplier]{Section~\ref{sec:multiplier}}. The multiplier bootstrap method, known for its effectiveness and promising performance, is the widely accepted approach to generating valid p-values. Expanding on this understanding, we introduce a multivariate Bernstein version of the multiplier bootstrap, based on the foundational work of~\citet{Lyu2023}. However, it is worth noting that the Bernstein multiplier bootstrap, while advantageous in various aspects, can be computationally intensive and does not readily accommodate the popular empirical beta copula.

Recent studies by~\citet{ Beare2020} and~\citet{ Seo2021} have explored a randomization test strategy to obtain feasible p-values. However, it is essential to recognize that this approach introduces a further computational burden compared to the multiplier bootstrap method. In response to these computational challenges, we propose an additional subsampling method for the empirical Bernstein copula process, with the potential to facilitate the generation of bootstraps for the empirical beta copula. Our simulation results indicate that these two procedures offer complementary advantages in addressing both computational and statistical aspects.

The remainder of the paper is organized as follows.~\hyperref[sec:2]{Section~\ref{sec:2}} presents the proposed test statistics and examines their asymptotic behaviours.~\hyperref[sec:multiplier]{Section~\ref{sec:multiplier}}  and~\hyperref[sec:subsam]{\ref{sec:subsam}} focus on the deployment of the multiplier bootstrap and subsampling method for the empirical Bernstein copula process, respectively. In~\hyperref[sec:simulation]{Section~\ref{sec:simulation}}, the multiplier bootstrap and subsampling method are implemented and a simulation study is carried out to assess the performance of the proposed approaches. The proofs are provided in \hyperref[app]{Appendix}.

\section{Testing procedure}\label{sec:2}

In this section, a test procedure is presented for independent and paired samples. Heuristically, it can be shown that the two cases can be treated similarly for testing two-sample homogeneity of the dependence structure.

\subsection{Independent samples}\label{subsec:1}

Consider two independent random samples of independent vectors valued with $\RR^d$, where $d$ is much smaller than the sample sizes $n_1$ and $n_2$. The first sample consists of vectors $(X_{11},\ldots, X_{1d}),\ldots,(X_{n_11}, \ldots, X_{n_1d})$ with cumulative distribution function $F$, which has continuous margins $F_1,\ldots, F_d$. The  second sample consists of vectors $(Y_{11},\ldots, Y_{1d}),\ldots, (Y_{n_21}, \ldots, Y_{n_2d})$ with cumulative distribution function $G$, associated with continuous margins $G_1,\ldots, G_d$. Let $C, D$ denote the underlying copula functions of the two random samples, by~\citet{Sklar1959}, for any $\bm{u}=(u_1, \ldots, u_d)\in [0, 1]^d$,
\begin{equation*}
	C(\bm{u})=F(F_1^{-1}(u_1), \ldots, F_d^{-1}(u_d)),\quad D(\bm{u})=G(G_1^{-1}(u_1), \ldots, G_d^{-1}(u_d)),
\end{equation*}
where $F_{\ell}^{-1}(u_{\ell})=\inf\{t\in \RR: F_{\ell}(t)\ge u_{\ell}\}, \ell=1,\ldots, d$, and $G_\ell^{-1}(\cdot)$ is defined similarly. To determine whether the dependence structures of the two samples are identical, the corresponding hypothesis test is
\begin{equation}{\label{eq: hypotheses}}
	\begin{cases}
		\mathscr{H}_0: C(\bm{u})=D(\bm{u}) & \forall \bm{u} \in [0,1]^d\\
		& \text{versus}\\
		\mathscr{H}_1:C(\bm{u})\ne D(\bm{u}) & \exists \bm{u} \in [0,1]^d .
	\end{cases}	 
\end{equation}	
Based on the two samples, the empirical copulas are defined as
\begin{equation*}
	C_{n_1}(\bm{u})=\frac{1}{n_1}\sum_{i=1}^{n_1}\prod_{\ell=1}^{d}\II\left(\widehat{U}_{i\ell}\le u_{\ell}\right),\quad 	D_{n_2}(\bm{u})=\frac{1}{n_2}\sum_{i=1}^{n_2}\prod_{\ell=1}^{d}\II\left(\widehat{V}_{i\ell}\le u_{\ell}\right),
\end{equation*}
where $\II(\cdot)$ denotes the indicator function and $\widehat{U}_{i\ell}, \widehat{V}_{i\ell} $ are the pseudo-observations with
\begin{align*}
	\widehat{U}_{i\ell}&=\frac{1}{n_1}\sum_{j=1}^{n_1}\II(X_{j\ell}\le X_{i\ell}), \quad i\in \{1, \ldots, n_1\}, \\	
	\widehat{V}_{i\ell}&=\frac{1}{n_2}\sum_{j=1}^{n_2}\II(Y_{j\ell}\le Y_{i\ell}), \quad i\in \{1, \ldots, n_2\}.
\end{align*}
As outlined in the introduction, empirical copulas exhibit certain undesired disadvantages. To address these issues, one can apply smoothing techniques on empirical copulas using Bernstein polynomials.  For a function $f$ defined on the closed interval $[0, 1]$, the expression
\begin{equation*}
	f_m(x)=\sum_{k=0}^{m}f\left(\frac{k}{m}\right)P_{k, m}(x),
\end{equation*}
where $P_{k, m}(x)=\binom{m}{k}x^k(1-x)^{m-k}$ is the binomial probability mass function,  is called the Bernstein polynomial of order $m\in \NN^+$, which is introduced by~\citet{Bernvstein1912demonstration}. For a more comprehensive treatment of Bernstein polynomials, readers are directed to the monograph by~\citet{Lorentz1986} and for pioneering applications of Bernstein polynomials in statistics, see \cite{Vitale1975}, \cite{Bouezmarni2005consistency}, \cite{BouezmarniRolin2007}, and \cite{Bouezmarni2013bernstein}. Specifically, the empirical Bernstein copulas are defined by, 
\begin{equation}\label{eq:2022-07-10, 11:01AM}
	\begin{aligned}
		C_{n_1, m_1}(\bm{u})&=\sum_{k_1=0}^{m_1}\cdots\sum_{k_d=0}^{m_1}C_{n_1}\left(\frac{k_1}{m_1}, \ldots, \frac{k_d}{m_1}\right)\prod_{\ell=1}^d P_{k_\ell, m_1}(u_\ell)\\
		&=\frac{1}{n_1}\sum_{i=1}^{n_1}\left[\sum_{k_1=0}^{m_1}\cdots\sum_{k_d=0}^{m_1}\prod_{\ell=1}^{d}\II\left(\widehat{U}_{i\ell}\le k_{\ell}/m_1\right)P_{k_{\ell}, m_1}(u_{\ell})\right],\\ 	
		D_{n_2, m_2}(\bm{u})&=\sum_{k_1=0}^{m_2}\cdots\sum_{k_d=0}^{m_2}D_{n_2}\left(\frac{k_1}{m_2}, \ldots, \frac{k_d}{m_2}\right)\prod_{\ell=1}^d P_{k_\ell, m_2}(u_\ell)\\
		&=\frac{1}{n_2}\sum_{i=1}^{n_2}\left[\sum_{k_1=0}^{m_2}\cdots\sum_{k_d=0}^{m_2}\prod_{\ell=1}^{d}\II\left(\widehat{V}_{i\ell}\le k_{\ell}/m_2\right)P_{k_{\ell}, m_2}(u_{\ell})\right],
	\end{aligned}
\end{equation}
where, for simplicity, one considers the same Bernstein order $m_1, m_2 \in \NN^+$ (which is assumed to be dependent on the sample size) for smoothing each dimension in the two samples. Note that, when the Bernstein order equals the sample size, the empirical Bernstein copula corresponds to the empirical beta copula, see~\citet{Segers2017} and~\citet{Kiriliouk2021}. To construct the equality test, Cram\'{e}r-von Mises statistics $\{R_{\bm{n}}^{\bm{m}}, S_{\bm{n}}^{\bm{m}}\}$ and Kolmogorov-Smirnov statistic $T_{\bm{n}}^{\bm{m}}$ are proposed relying on Equation~\eqref{eq:2022-07-10, 11:01AM}. Let $\bm{m}=(m_1, m_2), \bm{n}=(n_1, n_2)$, and  $\lambda_{\bm{n}}= n_1/(n_1+n_2)$, define
\begin{align}\label{eq:2023-01-24}
	R_{\bm{n}}^{\bm{m}}&=n_2\lambda_{\bm{n}} \int_{[0, 1]^d}\Big\{C_{n_1,m_1}(\bm{u})-D_{n_2,m_2}(\bm{u})\Big\}^2\dif\bm{u},\notag\\
	S_{\bm{n}}^{\bm{m}}&=n_2\lambda_{\bm{n}} \int_{[0, 1]^d}\Big\{{C}_{n_1, m_1}(\bm{u})-{D}_{n_2, m_2}(\bm{u})\Big\}^2\dif \mathscr{C}_{\bm{n}}^{\bm{m}}(\bm u),\\
	T_{\bm{n}}^{\bm{m}}&=\sqrt{n_2 \lambda_{\bm{n}} }\sup_{\bm{u}\in [0, 1]^d}|{C}_{n_1, m_1}(\bm{u})-{D}_{n_2, m_2}(\bm{u})|\notag,
\end{align}
where $\mathscr{C}_{\bm{n}}^{\bm{m}}$ is the pooled copula estimator defined as
\begin{equation*}
	\mathscr{C}_{\bm{n}}^{\bm{m}}(\bm u)=\frac{n_1{C}_{n_1, m_1}(\bm{u})+n_2D_{n_2, m_2}(\bm{u})}{n_1+n_2}.
\end{equation*} 
Before delving further, it is essential to highlight a distinction between the equality test outlined in~\cite{Bakam2024smooth} and our proposed approach based on Bernstein polynomials. Within the investigation conducted by~\cite{Bakam2024smooth}, the copula function undergoes a decomposition utilizing the Legendre polynomial basis for the construction of the equality test. However, in contrast, our proposed empirical Bernstein copula, while conceptualized as a decomposition into the Bernstein basis, is crucially non-orthogonal. It is also worth emphasizing that the Bernstein basis exhibits a pivotal variation-diminishing property, a feature not present in Legendre polynomials~\citep[as referenced in][]{Farouki2012bernstein}. More significantly, the test presented in~\cite{Bakam2024smooth} is characterized as Neyman's smooth test~\citep[see][]{Neyman1937smooth,Bera2013smooth,Zhou2017smoothtest}. Neyman's smooth test serves as a middle ground between omnibus and directional tests. It differs from our proposed empirical Bernstein copula-based omnibus test in this specific context.

For the establishment of asymptotic behaviors of the proposed test statistics, the following assumptions~\citep[see][]{Segers2012} are needed.
\begin{assumption}\label{ass:2022-07-05, 1:23PM}
	Assume that every partial derivative $\dot{C}_{\ell}(\bm{u})=\partial C(\bm{u})/\partial u_{\ell}, \dot{D}_{\ell}(\bm{u})=\partial D(\bm{u})/\partial u_{\ell}, \ell\in \{1,\ldots, d\}$ exists and is continuous in the set $V_{d,\ell}=\{\bm{u}\in [0, 1]^d: 0<u_{\ell}<1\}$.
\end{assumption}

\begin{assumption}\label{ass:2}
	For every $j, \ell\in \{1, \ldots, d\}$, the second-order partial derivative $\ddot{C}_{j\ell}(\bm{u})=\partial^2C(\bm{u})/\partial u_j \partial u_{\ell}$ and $\ddot{D}_{j\ell}(\bm{u})=\partial^2D(\bm{u})/\partial u_j \partial u_{\ell}$ are defined and continuous, on the sets $V_{d, j}\cap V_{d, \ell}$, and there exists constants $K_1, K_2>0$ such that
	\begin{align*}
		\left|\ddot{C}_{j\ell}(\bm{u})\right|&\le K_1\min\left(\frac{1}{u_j(1-u_j)}, \frac{1}{u_{\ell}(1-u_{\ell})}\right), \quad \bm{u}\in V_{d, j}\cap V_{d, \ell}\\
		\left|\ddot{D}_{j\ell}(\bm{u})\right|&\le K_2\min\left(\frac{1}{u_j(1-u_j)}, \frac{1}{u_{\ell}(1-u_{\ell})}\right), \quad \bm{u}\in V_{d, j}\cap V_{d, \ell}.
	\end{align*} 
\end{assumption}

Let $\ell^{\infty}\left([0, 1]^d\right)$ be the  space of real-valued, bounded functions on $[0, 1]^d$, equipped with the supremum norm $ \| \cdot \|_\infty $. The arrow $ \rightsquigarrow $ denotes weak convergence in the sense used in~\cite{Van_der_vaart1996}. Furthermore, let $\CC_{n_1}=\sqrt{n_1}\{C_{n_1}-C\}$ and $\DD_{n_2}=\sqrt{n_2}\{D_{n_2}-D\}$ be the empirical copula processes of the two samples. Under \hyperref[ass:2022-07-05, 1:23PM]{Assumption~\ref{ass:2022-07-05, 1:23PM}} and $m_1=cn_1^{\alpha}$ with $c>0, \alpha\ge 1$, the empirical Bernstein copula process~\citep{Segers2017} converges weakly in $\ell^{\infty}\left([0, 1]^d\right)$ to a $d$-dimensional Gaussian process, i.e.,
\begin{equation}\label{eq:2023-09-16, 11:13AM}
	\CC_{n_1,m_1} = \sqrt{n_1}\{C_{n_1, m_1}-C\}\rightsquigarrow\CC,
\end{equation}
where

\[
\CC(\bm{u}) =\BB_C(\bm{u})-\sum_{\ell=1}^{d}\BB_C(\bm{u}^{\ell})\dot{C}_{\ell}(\bm{u}), 
\]
with $\bm{u}^{\ell}=(1,\ldots, 1, u_{\ell}, 1,\ldots, 1)$, i.e.,  all elements of $\bm{u}$ are replaced by $1$ except the $\ell$-th component and $\BB_C$ is a Brownian bridge with covariance function
\begin{equation*}
	\EE\left[\BB_C(\bm{u})\BB_C(\bm{v})\right]=C(\bm{u}\wedge\bm{v})-C(\bm{u})C(\bm{v}),\qquad \bm{u}, \bm{v}\in [0, 1]^d.
\end{equation*}
Similarly, under \hyperref[ass:2022-07-05, 1:23PM]{Assumption~\ref{ass:2022-07-05, 1:23PM}} and $m_2=cn_2^{\alpha}$ with $c>0, \alpha\ge 1$,
\begin{equation}\label{eq:2024-10-07, 11:17AM}
	\DD_{n_2,m_2}=\sqrt{n_2}\{D_{n_2, m_2}-D\}\rightsquigarrow\DD,
\end{equation}
where 
\[
\DD(\bm{u})=\BB_D(\bm{u})-\sum_{\ell=1}^{d}\BB_D(\bm{u}^{\ell})\dot{D}_{\ell}(\bm{u}).
\]
Due to the weak convergence of the empirical Bernstein copula process, the following lemma is established.

\begin{lemma}\label{lem:2022-06-22, 12:05PM}
	Suppose that $\lambda_{\bm{n}}= n_1/(n_1+n_2)\to \lambda\in (0, 1)$ and $m_r=cn_r^{\alpha}$ with $c>0, \alpha\ge 1$ for $r=1, 2$. Additionally, assume that~\hyperref[ass:2022-07-05, 1:23PM]{Assumption~\ref{ass:2022-07-05, 1:23PM}} holds, then, as $\min(n_1, n_2)\to \infty$,
	\begin{equation}\label{eq:2022-07-08, 10:39PM}
		\FF^{\bm{m}}_{\bm{n}}=\sqrt{1-\lambda_{\bm{n}} }\CC_{n_1, m_1}-\sqrt{\lambda_{\bm{n}} }\DD_{n_2, m_2}\rightsquigarrow \FF=\sqrt{1-\lambda}\CC-\sqrt{\lambda}\DD
	\end{equation}
	in $\ell^{\infty}\left([0, 1]^d\right)$, where $\CC$ and $\DD$ are independent processes.
\end{lemma}

Under the same conditions of \hyperref[lem:2022-06-22, 12:05PM]{Lemma~\ref{lem:2022-06-22, 12:05PM}}, the weak convergence of the proposed test statistics is established as follows.
\begin{theorem}\label{thm:2022-06-22, 12:31PM}
	Suppose that the assumptions in \hyperref[lem:2022-06-22, 12:05PM]{Lemma~\ref{lem:2022-06-22, 12:05PM}} are satisfied. Then, under the null hypothesis, as $\min(n_1, n_2)\to \infty$,
	\begin{align}
		R_{\bm{n}}^{\bm{m}}&=\int_{[0, 1]^d}\Big\{\FF^{\bm{m}}_{\bm{n}}(\bm{u})\Big\}^2\dif\bm{u}\rightsquigarrow R_{\, \FF}=\int_{[0, 1]^d}\{\FF(\bm{u})\}^2\dif\bm{u},\notag\\
		S_{\bm{n}}^{\bm{m}}&=\int_{[0, 1]^d}\Big\{\FF^{\bm{m}}_{\bm{n}}(\bm{u})\Big\}^2\dif	\mathscr{C}_{\bm{n}}^{\bm{m}}(\bm u)\rightsquigarrow S_{\, \FF}= \int_{[0, 1]^d}\{\FF(\bm{u})\}^2\dif C_{}(\bm{u}),\notag\\
		T_{\bm{n}}^{\bm{m}}&=\sup_{\bm{u}\in [0, 1]^d}|\FF^{\bm{m}}_{\bm{n}}(\bm{u})|\rightsquigarrow T_{\, \FF}=\sup_{\bm{u}\in [0, 1]^d}|\FF_{}(\bm{u})|.
	\end{align}
\end{theorem}

More generally, the asymptotic behaviour of the test statistics under the alternative hypothesis is given as follows.

\begin{theorem}\label{thm:2022-06-23, 3:40AM}
	Suppose that the assumptions in \hyperref[lem:2022-06-22, 12:05PM]{Lemma~\ref{lem:2022-06-22, 12:05PM}} are satisfied. Then, as $\min(n_1, n_2) \rightarrow \infty$,
	\begin{align*}
		n_2^{-1}R_{\bm{n}}^{\bm{m}}&\xrightarrow{a.s.} R(C, D)=\lambda\int_{[0, 1]^d}\Big\{C(\bm{u}
		)-D(\bm{u})\Big\}^2\dif\bm{u}\\
		n_2^{-1}S_{\bm{n}}^{\bm{m}}&\xrightarrow{a.s.} S(C, D)=\lambda\int_{[0, 1]^d}\Big\{C(\bm{u}
		)-D(\bm{u})\Big\}^2\dif \left(\lambda C(\bm{u})+(1-\lambda) D(\bm{u})\right),\\
		\sqrt{n_2^{-1}}T_{\bm{n}}^{\bm{m}}&\xrightarrow{a.s.} T(C, D)=\sqrt{\lambda}\sup_{\bm{u}\in [0, 1]^d}|C(\bm{u}
		)-D(\bm{u})|.
	\end{align*}
	Specifically, under any fixed alternative hypothesis, as $\min(n_1, n_2)\to \infty$, 
	\begin{equation*}
		R_{\bm{n}}^{\bm{m}}\xrightarrow{a.s.}\infty,  \quad 	S_{\bm{n}}^{\bm{m}}\xrightarrow{a.s.}\infty ,\quad 	T_{\bm{n}}^{\bm{m}}\xrightarrow{a.s.}\infty, 
	\end{equation*}
	which provides a guarantee of the consistency of the proposed test.
	
\end{theorem}

\begin{remark}
	Note that, as a consequence of the following \hyperref[lem:2023-09-16, 7:59AM]{Lemma~\ref{lem:2023-09-16, 7:59AM}}, the convergences in \hyperref[thm:2022-06-22, 12:31PM]{Theorem~\ref{thm:2022-06-22, 12:31PM}-\ref{thm:2022-06-23, 3:40AM}} also hold under \hyperref[ass:2022-07-05, 1:23PM]{Assumptions~\ref{ass:2022-07-05, 1:23PM}-\ref{ass:2}} with a weaker condition on the Bernstein order, that is, $m_r = cn_r^{\alpha}$ with $c > 0, \alpha > 3/4$ for $r=1, 2$. 
\end{remark}

\begin{lemma}\label{lem:2023-09-16, 7:59AM}
	Suppose that $C$ satisfies \hyperref[ass:2022-07-05, 1:23PM]{Assumptions~\ref{ass:2022-07-05, 1:23PM}-\ref{ass:2}}, and suppose that $m_r = cn_r^{\alpha}$ with $c > 0, \alpha > 3/4$ for $r=1, 2$, then, in $\ell^{\infty}\left([0, 1]^d\right)$, Equation~\eqref{eq:2023-09-16, 11:13AM} is valid.
\end{lemma}

Imposing more conditions than that of~\hyperref[thm:2022-06-22, 12:31PM]{Theorem 1}, we establish a strong approximation result as given in the following theorem.

\begin{theorem}\label{thm:2023-02-05, 8:49PM}
	Suppose that $\lambda_{\bm{n}}\to \lambda$ and $m_r=cn_r^{\alpha}$ with $c>0, \alpha \ge 1$ for $r=1, 2$ as $\min(n_1, n_2)\to \infty$. Furthermore, suppose that \hyperref[ass:2022-07-05, 1:23PM]{Assumptions~\ref{ass:2022-07-05, 1:23PM}-\ref{ass:2}} are satisfied. On a suitable probability space~\citep[Remark 2.1]{Bouzebda&ElFaouzi&Zari2011}, it is possible to define $\FF_{\bm{n}}^{\bm{m}}(\bm{u}), \bm{u}\in [0, 1]^d$, jointly with Gaussian processes $\{\FF^*_{\bm{n}}(\bm{u})\}_{\bm{n}\in \NN^+\times\, \NN^+}$, in such a way that, under the null hypothesis, as $\min(n_1, n_2)\to \infty$, we have
	\begin{align*}
		\left|R_{\bm{n}}^{\bm{m}}-\int_{[0, 1]^d}\{\FF^*_{\bm{n}}(\bm{u})\}^2\dif\bm{u}\right|&\overset{a.s.}{=}O\left(\max\left((\log\log n_1)^{1/2}, (\log \log n_2)^{1/2}\right)\Psi(n_1, n_2)\right),\\
		\left|S_{\bm{n}}^{\bm{m}}-\int_{[0, 1]^d}\{\FF^*_{\bm{n}}(\bm{u})\}^2\dif C(\bm{u}) \right|&\overset{a.s.}{=}O\left(\max\left((\log\log n_1)^{1/2}, (\log \log n_2)^{1/2}\right)\Psi(n_1, n_2)\right),\\
		\left|T_{\bm{n}}^{\bm{m}}-\sup_{\bm{u}\in [0, 1]^d}|\FF^*_{\bm{n}}(\bm{u})|\right|&\overset{a.s.}{=}O\left(\Psi(n_1, n_2)\right),
	\end{align*}
	where 
	\begin{equation*}
		\Psi(n_1, n_2)=\max\left(n_1^{-1/(2(2d-1))}\log n_1, n_2^{-1/(2(2d-1))}\log n_2\right),
	\end{equation*}
	and
	\begin{equation*}
		\FF^*_{\bm{n}}=\sqrt{1-\lambda_{\bm{n}} }\CC^{(n_1)}+\sqrt{\lambda_{\bm{n}} }\DD^{(n_2)}
	\end{equation*}
	with $\{\CC^{(n_1)}\}_{n_1\in \NN^+} \,(\text{resp. } \{\DD^{(n_2)}\}_{n_2\in \NN^+})$ are independent copies of $\CC \,(\text{resp. }\DD)$. Additionally, $\{\CC^{(n_1)}\}_{n_1\in \NN^+}$ are also independent of $\{\DD^{(n_2)}\}_{n_2\in \NN^+}$.
\end{theorem}

\subsection{Paired samples}

\n Sometimes, one may encounter paired samples. This subsection addresses a testing procedure in this scenario. Consider independent vectors $\bm{Z}_1=(X_{11},\ldots, X_{d1}, Y_{11},\ldots, Y_{d1}), \ldots, \\
\bm{Z}_n=(X_{1n},\ldots, X_{dn},Y_{1n},\ldots, Y_{dn})$ with copula function $\mathcal{C}$ on $[0, 1]^{2d}$, satisfying the property that for any $\bm{u}, \bm{v}\in [0, 1]^d$,
\begin{equation*}
	\C(\bm{u}, \underbrace{1,\ldots, 1}_{d})=C(\bm{u}),\quad \C(\underbrace{1, \ldots, 1}_d, \bm{v})=D(\bm{v}),
\end{equation*}
where $C, D$ are copulas for $X$ and $Y$, respectively. The objective is to investigate if $C$ and $D$ are equal. Define the empirical Bernstein copula for sample $Z$ as follows, for $\bm{u}, \bm{v}\in [0, 1]^d$,
\begin{align*}
	\C_{n, m}(\bm{u}, \bm{v})&=\frac{1}{n}\sum_{i=1}^{n}\left[\sum_{k_1=0}^{m}\cdots\sum_{k_d=0}^m\sum_{\kappa_{1}=0}^{m}\cdots\sum_{\kappa_{d}=0}^{m}\prod_{\ell=1}^{d}\prod_{j=1}^{d} \II\left(\widehat{U}_{i\ell}\le k_{\ell}/m\right)\right.\\
	&\quad \cdot \left.\II\left(\widehat{V}_{ij}\le \kappa_{j}/m\right)P_{k_{\ell}, m}(u_{\ell})P_{\kappa_{j}, m}(v_{j})\right].
\end{align*}
Furthermore, under \hyperref[ass:2022-07-05, 1:23PM]{Assumption 1} and as $n/m \to c>0$, let $\bm{z}\in [0, 1]^{2d}$. As $n\to \infty$, in the space $\ell^{\infty}\left([0, 1]^{2d}\right)$,
\begin{equation*}
	\mathfrak{C}_{n, m}=\sqrt{n}\{\C_{n, m}-\C\}\rightsquigarrow\mathfrak{C},
\end{equation*}
where
\begin{equation*}
	\mathfrak{C}(\bm{z}) =\BB_\C(\bm{z})-\sum_{\ell=1}^{2d}\BB_\C(\bm{z}^{\ell})\dot{\C}_{\ell}(\bm{z}).
\end{equation*}
It is readily observed that, assuming $m$ is a divisor of $n$, then
\begin{equation}\label{eq:2023-01-31, 8:58AM}
	\C_{n, m}(\bm{u},1,\ldots, 1)=C_{n, m}(\bm{u}),\quad \C_{n, m}(1, \ldots, 1, \bm{v})=D_{n, m}(\bm{v}),
\end{equation}
where $C_{n, m}(\bm{u}), D_{n, m}(\bm{v})$ are defined in Equation~\eqref{eq:2022-07-10, 11:01AM}, i.e., the empirical Bernstein copulas for $X$ and $Y$, respectively. 
It immediately follows that, as $n\to \infty$, in the space $\ell^{\infty}\left([0, 1]^d\right)$,
\begin{align*}
	\CC_{n, m}&=\sqrt{n}\{C_{n, m} - C \}\rightsquigarrow \CC,\\ 
	\DD_{n, m} &=\sqrt{n}\{D_{n, m} - D\}\rightsquigarrow \DD.
\end{align*}
Hence, the case of paired samples can be treated the same as the case of independent samples, by separating $X$ and $Y$ as in Equation~\eqref{eq:2023-01-31, 8:58AM}. Therefore, all procedures and results for independent samples are valid for paired samples as well. For this reason, the remainder of this article focuses on the case of independent samples.

\section{Multiplier bootstrap}\label{sec:multiplier}

\n The standard Monte Carlo procedure is unsuitable for this scenario due to its reliance on the underlying copula $C$ and $D$, thus rendering it inadequate for accurately evaluating the limits of the proposed test statistics. To address this limitation, a widely recognized approach in the literature is the multiplier bootstrap method, originally introduced by~\citet{Scaillet2005a} and further advanced by~\citet{Remillard2009}. Prominently, this method has garnered significant attention and adoption by numerous researchers~\citep{Kojadinovic2011, Genest2012, Harder2017, Bahraoui2018} across diverse research contexts. The multiplier bootstrap technique offers notable benefits, including the provision of valid p-values and exceptional efficiency, particularly when applied to goodness-of-fit tests.  In particular, ~\cite{Kojadinovic2024resampling} introduced a sequential multiplier bootstrap method for a class of smooth empirical copulas.

Recently,~\citet{Lyu2023} constructed multiplier bootstraps for the empirical Bernstein copula process dealing with bivariate copulas. Here, we extend these results to the more general case. Let $H\in \NN^+$ and for each $h\in\{1,\ldots, H\}$, let $\bm{\xi}_{n_1+n_2}^{(h)}\coloneqq\left(\xi_1^{(h)},\ldots, \xi_{n_1+n_2}^{(h)}\right)$ be a vector of independent random variables with unit mean and unit variance (we simply take $\xi_i^{(h)}\sim \textsf{Exponential}(1), i=1, \ldots, n_1+n_2$). Furthermore, $\{\xi_{n_1+n_2}^{(h)}\}$ are independent across $h$. Denote the mean of the sample of the first $n_1$ observations by $\overline{\xi}_{n_1}^{(h)}$, and the sample mean of the remaining $n_2$ observations by $\overline{\xi}_{n_2}^{(h)}$.  Set
\begin{align*}
	G_{n_1, m_1}(\bm{u})&=\frac{1}{{n_1}}\sum_{i=1}^{n_1}\left[\sum_{k_1=0}^{m_1}\cdots\sum_{k_d=0}^{m_1}\prod_{\ell=1}^{d}\II\left({U}_{i\ell}\le k_{\ell}/m_1\right)P_{k_{\ell}, m_1}(u_{\ell})\right]\\
	&=\sum_{k_1=0}^{m_1}\cdots\sum_{k_d=0}^{m_1}G_{n_1}\left(\frac{k_1}{m_1},\ldots, \frac{k_d}{m_1}\right)\prod_{\ell=1}^{d}P_{k_{\ell}, m_1}(u_{\ell}),
\end{align*}
where $U_{i\ell}=F_{\ell}(X_i), \ell=1, \ldots, d$. For $h\in \{1, \ldots, H\}$, define
\begin{align*}
	{\GG}_{n_1, m_1}(\bm{u})&=\sqrt{n_1}\left\{G_{n_1, m_1}(\bm{u})-C(\bm{u})\right\}\\
	&=\frac{1}{\sqrt{n_1}}\sum_{i=1}^{n_1}\left[\sum_{k_1=0}^{m_1}\cdots\sum_{k_d=0}^{m_1}\prod_{\ell=1}^{d}\II\left({U}_{i\ell}\le k_{\ell}/m_1\right)P_{k_{\ell}, m_1}(u_{\ell})-C(\bm{u})\right]	,
\end{align*}
and
\begin{equation*}
	{\GG}^{(h)}_{n_1, m_1}(\bm{u})=\frac{1}{\sqrt{n_1}}\sum_{i=1}^{n_1}\left(\xi_i^{(h)}-\overline{\xi}^{(h)}_{n_1}\right)\left[\sum_{k_1=0}^{m_1}\cdots\sum_{k_d=0}^{m_1}\prod_{\ell=1}^{d}\II\left(\widehat{U}_{i\ell}\le k_{\ell}/m_1\right)P_{k_{\ell}, m_1}(u_{\ell})\right].	
\end{equation*}
Note that ${\GG}_{n_1, m_1}$ is the empirical Bernstein copula process assuming the margins are known. Define $ G_{n_2, m_2}(\bm{u}),\GG_{n_2, m_2}(\bm{u})$, and $\GG_{n_2, m_2}^{(h)}(\bm{u}) $ similarly for the second sample. The following proposition states that $\GG_{n_1, m_1}, \GG_{n_2, m_2}$ converge weakly to the Brownian bridge mentioned in the preceding section. Moreover, the replicates converge to independent copies of the Brownian bridges.

\begin{proposition}\label{prop:2023-02-10, 7:47AM}
	Suppose that $\lambda_{\bm{n}}\to \lambda$ and $  m_r=cn_r^{\alpha}$ with $c>0, \alpha>3/4$ for $r=1, 2$ as $\min(n_1, n_2)\to \infty$. Furthermore, suppose that \hyperref[ass:2022-07-05, 1:23PM]{Assumptions~\ref{ass:2022-07-05, 1:23PM}-\ref{ass:2}} are satisfied. Then, in $\ell^{\infty}\left([0,1]^d\right)$, 
	\begin{enumerate}[label=(\roman*)]
		\item   one has 
		\begin{equation*}
			{\GG}_{n_1, m_1}	\rightsquigarrow \BB_C ,\qquad {\GG}_{n_2, m_2}	\rightsquigarrow \BB_D,
		\end{equation*}
		
		\item for each $h\in \{1,\ldots, H\}$, 
		\begin{equation*}
			{\GG}^{(h)}_{n_1, m_1}	\rightsquigarrow \BB^{(h)}_C,\qquad {\GG}^{(h)}_{n_2, m_2}	\rightsquigarrow \BB^{(h)}_D,
		\end{equation*}
		where $\{\BB^{(h)}_C, \BB^{(h)}_D\}_{h=1}^H$ are independent copies of $\{ \BB_C, \BB_D\}$.
	\end{enumerate}
\end{proposition}

Therefore, bootstrap replicates of $\CC_{n_1, m_1}$ for $h\in\{1, \ldots, H\}$ can be defined as 
\begin{equation*}
	\CC_{n_1, m_1}^{(h)}(\bm{u})={\GG}^{(h)}_{n_1, m_1}(\bm{u})-\sum_{\ell=1}^d{\GG}^{(h)}_{n_1, m_1}(\bm{u}^{\ell})\,\partial C_{n_1, m_1}(\bm{u})/\partial u_\ell,
\end{equation*}
where
\begin{align*}
	\partial C_{n_1, m_1}(\bm{u})/\partial u_\ell&=m_1\sum_{k_1=0}^{m_1}\cdots \sum_{k_{\ell}=0}^{m_1-1}\cdots \sum_{k_{d}=0}^{m_1}\left\{C_{n_1}\left(\frac{k_1}{m_1},\ldots, \frac{k_{\ell}+1}{m_1},\ldots, \frac{k_d}{m_1}\right)\right.\\
	&\quad -\left. C_{n_1}\left(\frac{k_1}{m_1},\ldots, \frac{k_{\ell}}{m_1},\ldots, \frac{k_d}{m_1}\right)\right\}\\
	&\quad \cdot P_{k_1, m_1}(u_1)\cdots P_{k_{\ell}, m_1-1}(u_{\ell})\cdots P_{k_d, m_1}(u_d).
\end{align*}
Similarly, the replicates of $\DD_{n_2, m_2}$ are defined as  
\begin{equation*}
	\DD_{n_2, m_2}^{(h)}(\bm{u})={\GG}^{(h)}_{n_2, m_2}(\bm{u})-\sum_{\ell=1}^d{\GG}^{(h)}_{n_2, m_2}(\bm{u}^{\ell})\, \partial D_{n_2, m_2}(\bm{u})/\partial u_{\ell}.
\end{equation*}

The partial derivatives of the empirical Bernstein copula for the bivariate case were studied in~\citet{Janssen2016}. Unlike the empirical copula, these empirical Bernstein copula partial derivatives can be calculated directly without any further estimation. The following proposition shows the uniform consistency of these partial derivatives.

\begin{proposition}\label{prop:2023-02-10, 7:51AM}
	Suppose that $\lambda_{\bm{n}}\to \lambda$ and $  m_r=cn_r^{\alpha}$ with $c>0, \alpha>3/4$ for $r=1, 2$ as $\min(n_1, n_2)\to \infty$. In addition, \hyperref[ass:2022-07-05, 1:23PM]{Assumptions~\ref{ass:2022-07-05, 1:23PM}-\ref{ass:2}} are satisfied. Then for any $\ell\in \{1, \ldots, d\}$ and $b_{n_r}=\kappa n_r^{-\beta}$ with $\kappa>0, \beta < \alpha/2$, as $\min(n_1, n_2)\to \infty$,
	\begin{align*}
		\sup_{\{\bm{u}\in [0, 1]^d: u_\ell\in [b_{n_1}, 1-b_{n_1}]\}}\left|\partial C_{n_1, m_1}(\bm{u})/\partial u_\ell-\dot{C}_{\ell}(\bm{u})\right|&\overset{a.s.}{=}O\left(m_1^{1/2}n_1^{-1/2}(\log \log n_1)^{1/2}\right),\\
		\sup_{\{\bm{u}\in [0, 1]^d: u_\ell\in [b_{n_2}, 1-b_{n_2}]\}}\left|\partial D_{n_2, m_2}(\bm{u})/\partial u_\ell-\dot{D}_{\ell}(\bm{u})\right|&\overset{a.s.}{=}O\left(m_2^{1/2}n_2^{-1/2}(\log \log n_2)^{1/2}\right).
	\end{align*}
\end{proposition}

Combining~\hyperref[prop:2023-02-10, 7:47AM]{Propositions 1-2}, for $\bm{u}\in (0, 1)^d$, $\CC_{n_1, m_1}^{(h)}(\bm{u}), 	\DD_{n_2, m_2}^{(h)}(\bm{u})$ are valid replicates for $\CC_{n_1, m_1}^{}(\bm{u}), 	\DD_{n_2, m_2}^{}(\bm{u})$. Further, plugging $\CC_{n_1, m_1}^{(h)}(\bm{u}), 	\DD_{n_2, m_2}^{(h)}(\bm{u})$ into Equation~\eqref{eq:2022-07-08, 10:39PM}, one has 
\begin{equation}
	\FF^{\bm{m}, (h)}_{\bm{n}}=\sqrt{1-\lambda_{\bm{n}} }\CC^{(h)}_{n_1, m_1}-\sqrt{\lambda_{\bm{n}} }\DD^{(h)}_{n_2, m_2}.
\end{equation}
Under same conditions of~\hyperref[thm:2022-07-08, 10:37PM]{Theorem 4}, one has
\begin{equation*}
	\left(	\FF_{\bm{n}}^{\bm{m}},\FF_{\bm{n}}^{\bm{m}, (1)}, \ldots, \FF_{\bm{n}}^{\bm{m}, (H)}\right)\rightsquigarrow \left(\FF, \FF^{(1)}, \ldots, \FF^{(H)}\right),
\end{equation*}
where $\{\FF^{(1)}, \ldots, \FF^{(H)}\}$ are independent copies of $\FF$. Therefore, for each $h \in \{1, \ldots, H\}$, we define replicates of the proposed test statistics as follows:
\begin{align*}
	R_{\bm{n}}^{\bm{m}, (h)} &= \int_{[0, 1]^d} \left\{\FF^{\bm{m}, (h)}_{\bm{n}}(\bm{u})\right\}^2 \, \dif\bm{u}, \\
	S_{\bm{n}}^{\bm{m}, (h)} &= \int_{[0, 1]^d} \left\{\FF^{\bm{m}, (h)}_{\bm{n}}(\bm{u})\right\}^2 \, \dif \mathscr{C}_{\bm{n}}^{\bm{m}}(\bm u), \\
	T_{\bm{n}}^{\bm{m}, (h)} &= \sup_{\bm{u} \in [0, 1]^d} \left|\FF^{\bm{m}, (h)}_{\bm{n}}(\bm{u})\right|.
\end{align*}
Moreover, the p-values for the test statistics can be approximated as
\begin{equation*}
	p_{\bm{n}, R}=\frac{1}{H}\sum_{h=1}^{H}\II\left(R_{\bm{n}}^{\bm{m}, (h)}> R_{\bm{n}}^{\bm{m}}\right),\quad p_{\bm{n}, S}=\frac{1}{H}\sum_{h=1}^{H}\II\left(S_{\bm{n}}^{\bm{m}, (h)}> S_{\bm{n}}^{\bm{m}}\right),
    \end{equation*}
    and
    \begin{equation*}
    p_{\bm{n}, T}=\frac{1}{H}\sum_{h=1}^{H}\II\left(T_{\bm{n}}^{\bm{m}, (h)}> T_{\bm{n}}^{\bm{m}}\right).
\end{equation*}
We establish the asymptotic properties of the multiplier replicates aforementioned and the limits of the approximated p-values in the following theorem.

\begin{theorem}\label{thm:2022-07-08, 10:37PM}
	Suppose that $\lambda_{\bm{n}}\to \lambda$ and $m_r=cn_r^{\alpha} $ with $c>0, 3/4< \alpha <1$ for $r=1, 2$ as $\min(n_1, n_2)\to \infty$ and suppose that~\hyperref[ass:2022-07-05, 1:23PM]{Assumptions~\ref{ass:2022-07-05, 1:23PM}-\ref{ass:2}} are satisfied. Then 
	\begin{enumerate}[label=(\roman*)]
		\item under the null hypothesis, as $\min(n_1, n_2)\to \infty$, for all $H\in \NN^+$,
		\begin{align*}
			\left(R_{\bm{n}}^{\bm{m}}, R_{\bm{n}}^{\bm{m}, (1)}, \ldots, R_{\bm{n}}^{\bm{m}, (H)}\right)&\rightsquigarrow \left(R_{\, \FF}, R_{\, \FF}^{(1)}, \ldots, R_{\, \FF}^{(H)}\right),\\
			\left(S_{\bm{n}}^{\bm{m}}, S_{\bm{n}}^{\bm{m}, (1)}, \ldots, S_{\bm{n}}^{\bm{m}, (H)}\right)&\rightsquigarrow \left(S_{\, \FF}, S_{\, \FF}^{(1)}, \ldots, S_{\, \FF}^{(H)}\right),\\
			\left(T_{\bm{n}}^{\bm{m}}, T_{\bm{n}}^{\bm{m}, (1)}, \ldots, T_{\bm{n}}^{\bm{m}, (H)}\right)&\rightsquigarrow \left(T_{\, \FF}, T_{\, \FF}^{(1)}, \ldots, T_{\, \FF}^{(H)}\right),
		\end{align*}
		where $\{R_{\, \FF}^{(h)}\}_{h=1}^H$,  $\{S_{\, \FF}^{(h)}\}_{h=1}^H$ and $\{T_{\, \FF}^{(h)}\}_{h=1}^H$ are independent copies of $R_{\, \FF}$, $S_{\, \FF}$ and $T_{\, \FF}$, respectively.	
		\item under any fixed alternative hypothesis, we have,
		\[
		\lim\limits_{\min(n_1, n_2) \to \infty}p_{\bm{n}, R}= \lim\limits_{\min(n_1, n_2) \to \infty}p_{\bm{n}, S}=\lim\limits_{\min(n_1, n_2) \to \infty}p_{\bm{n}, T}= 0.
		\]
	\end{enumerate}
\end{theorem}

It should be noted that the conditions imposed on the Bernstein order reduce to $m_r=cn_r^{\alpha} $ with $c>0, 3/4< \alpha <1$ for $r=1, 2$ as $\min(n_1, n_2)\to \infty$. This particular condition renders the application of the proposed multiplier bootstrap method infeasible for the empirical beta copula for which $m_r=n_r$. In the following section, an alternative resampling technique will be introduced that empowers interested readers to explore its potential application with empirical beta copula.

\section{Subsampling method}\label{sec:subsam}

\n An alternative to the multiplier bootstrap method, which is easier to implement, is the subsampling technique developed by~\citet{Kojadinovic2019}. This approach can be flexibly adapted to various smooth and weighted empirical copula processes, as demonstrated in the original paper. Unlike the multiplier method, it does not require the estimation of partial derivatives. Therefore, this method is valid for the empirical beta copula. In the subsequent section, we demonstrate the feasibility of applying this subsampling method to the empirical Bernstein copula process. In particular, this procedure offers performance comparable to the multiplier method, while demanding significantly fewer computational resources and relying on fewer theoretical assumptions.

Sample $b_1<n_1, b_2<n_2$ observations from $X$ and $Y$ without replacement, respectively. In this way, the number of possible subsamples are $N_{b_1, n_1}=\binom{n_1}{b_1}, N_{b_2, n_2}=\binom{n_2}{b_2}$. Let $h_1 \in \{1, \ldots, N_{b_1, n_1}\}$ and $h_2 \in \{1, \ldots, N_{b_2, n_2}\}$ be indices for the subsamples of $X$ and $Y$, respectively. Let $m_{(r)}$ for $r = 1, 2$ denote the Bernstein orders corresponding to these subsamples. Then, define the Bernstein polynomial estimators for each subsample as follows,
\begin{equation*}
	\begin{aligned}
		C^{(h_1)}_{b_1, m_{(1)}}(\bm{u})&=\frac{1}{b_{1, h_1}}\sum_{i=1}^{b_{1, h_1}}\left[\sum_{k_1=0}^{m_{(1)}}\cdots\sum_{k_d=0}^{m_{(1)}}\prod_{\ell=1}^{d}\II\left(\widehat{U}_{i\ell}\le k_{\ell}/m_{(1)}\right)P_{k_{\ell}, m_{(1)}}(u_{\ell})\right],\\ 	
		D^{(h_2)}_{b_2, m_{(2)}}(\bm{u})&=\frac{1}{b_{2, h_2}}\sum_{i=1}^{b_{2, h_2}}\left[\sum_{k_1=0}^{m_{(2)}}\cdots\sum_{k_d=0}^{m_{(2)}}\prod_{\ell=1}^{d}\II\left(\widehat{V}_{i\ell}\le k_{\ell}/m_{(2)}\right)P_{k_{\ell}, m_{(2)}}(u_{\ell})\right].
	\end{aligned}
\end{equation*}
The corrected subsample replicates of the empirical Bernstein copula process can be defined according to these subsample replicates,
\begin{align*}
	\CC_{b_1, m_{(1)}}^{(h_1)}&=\sqrt{\frac{b_1}{1-b_1/n_1}}\{C^{(h_1)}_{b_1, m_{(1)}}-C_{n_1, m_{1}}\},\\
	\DD_{b_2, m_{(2)}}^{(h_2)}&=\sqrt{\frac{b_2}{1-b_2/n_2}}\{D^{(h_2)}_{b_2, m_{(2)}}-D_{n_2, m_{2}}\}.
\end{align*}
The validity of these replicates is provided in the following theorem.

\begin{theorem}\label{thm:2022-07-10, 10:27PM}
	Suppose that the assumptions in \hyperref[lem:2022-06-22, 12:05PM]{Lemma~\ref{lem:2022-06-22, 12:05PM}} are satisfied and that $b_1, b_2$ go to infinity as $\min(n_1, n_2)\to \infty$. Let $\{h_{1, 1}, \ldots, h_{1, H}\} \subset  \{1, \ldots, N_{b_1, n_1}\}$ and $\{h_{2, 1}, \ldots, h_{2, H}\} \subset  \{1, \ldots, N_{b_2, n_2}\}$ be subsets of indices corresponding to the subsamples of $X$ and $Y$, respectively. Then
	\begin{equation*}
		\left(\CC_{n_1,m_1}, \CC_{b_1, m_{(1)}}^{(h_{1, 1})}, \ldots, \CC_{b_1, m_{(1)}}^{(h_{1, H})}, \DD_{n_2,m_2}, \DD_{b_2, m_{(2)}}^{(h_{2, 1})}, \ldots, \DD_{b_2, m_{(2)}}^{(h_{2, H})}\right)\rightsquigarrow\left(\CC, \CC^{(1)}, \ldots, \CC^{(H)}, \DD, \DD^{(1)}, \ldots, \DD^{(H)}\right)
	\end{equation*}
	in the space $\ell^{\infty}\left([0, 1]^d\right)^{\otimes (2H+2)}$, where $\{\CC^{(1)}, \ldots, \CC^{(H)}\}$ and $\{\DD^{(1)}, \ldots, \DD^{(H)}\}$ are independent copies of $\CC$ and $\DD$, respectively.
\end{theorem}

Let $I_1, \ldots, I_H$ be independent bivariate random vectors, where for each $h = 1, \ldots, H$, the components of $I_h$ are independent of $X$ and $Y$ and uniformly distributed on the set $\{1, \ldots, N_{b_1,n_1}\} \times \{1, \ldots, N_{b_2,n_2}\}$. Let $\tilde{\bm{m}} \coloneqq (m_{(1)}, m_{(2)})$ and $\bm{b} \coloneqq (b_1, b_2)$, substituting $\CC_{b_1, m_{(1)}}^{(h_1)}(\bm{u})$ and $\DD_{b_2, m_{(2)}}^{(h_2)}(\bm{u})$ into Equation~\eqref{eq:2022-07-08, 10:39PM}, one has
\begin{equation*}
	\FF^{\tilde{\bm{m}}, (I_h)}_{\bm{b}}=\sqrt{1-\lambda_{\bm{n}} }\CC_{b_1, m_{(1)}}^{(h_1)}-\sqrt{\lambda_{\bm{n}} }\DD_{b_2, m_{(2)}}^{(h_2)}.
\end{equation*}
Then
\begin{equation*}
	\left(	\FF_{\bm{n}}^{\bm{m}},\FF_{\bm{b}}^{\tilde{\bm{m}}, (I_1)}, \ldots, \FF_{\bm{b}}^{\tilde{\bm{m}}, (I_H)}\right)\rightsquigarrow \left(\FF, \FF^{(I_1)}, \ldots, \FF^{(I_H)}\right),
\end{equation*}
where $\{\FF^{(I_1)}, \ldots, \FF^{(I_H)}\}$ are independent copies of $\FF$. Using these replicates one can readily follow the strategies in \hyperref[sec:multiplier]{Section~\ref{sec:multiplier}} to construct valid replicates of the three proposed statistics. For each $h \in \{1, \ldots, H\}$, define
\begin{align*}
	R_{\bm{b}}^{\tilde{\bm{m}}, (I_h)} &= \int_{[0, 1]^d} \left\{\FF^{\tilde{\bm{m}}, (I_h)}_{\bm{b}}(\bm{u})\right\}^2 \, \dif\bm{u}, \\
	S_{\bm{b}}^{\tilde{\bm{m}}, (I_h)} &= \int_{[0, 1]^d} \left\{\FF^{\tilde{\bm{m}}, (I_h)}_{\bm{b}}(\bm{u})\right\}^2 \, \dif\left(\frac{b_1 C_{b_1, m_{(1)}}(\bm{u}) + b_2 D_{b_2, m_{(2)}}(\bm{u})}{b_1 + b_2}\right), \\
	T_{\bm{b}}^{\tilde{\bm{m}}, (I_h)} &= \sup_{\bm{u} \in [0, 1]^d} \left|\FF^{\tilde{\bm{m}}, (I_h)}_{\bm{b}}(\bm{u})\right|.
\end{align*}
Analogous to the multiplier method, the p-values for the test statistics can be approximated as
\begin{equation*}
	\tilde{p}_{\bm{n}, R}=\frac{1}{H}\sum_{h=1}^{H}\II\left(R_{\bm{b}}^{\tilde{\bm{m}}, (I_h)}> R_{\bm{n}}^{\bm{m}}\right),\quad \tilde{p}_{\bm{n}, S}=\frac{1}{H}\sum_{h=1}^{H}\II\left(S_{\bm{b}}^{\tilde{\bm{m}}, (I_h)}> S_{\bm{n}}^{\bm{m}}\right),
\end{equation*}
and
\begin{equation*}
    \tilde{p}_{\bm{n}, T}=\frac{1}{H}\sum_{h=1}^{H}\II\left(T_{\bm{b}}^{\tilde{\bm{m}}, (I_h)}> T_{\bm{n}}^{\bm{m}}\right).
\end{equation*}
Based on the proof of~\hyperref[thm:2022-07-08, 10:37PM]{Theorem~\ref{thm:2022-07-08, 10:37PM}}, the following theorem demonstrates the weak convergence of the proposed statistics and their corresponding replicates.

\begin{theorem}\label{thm:2024-08-02, 6:59PM}
	Suppose that assumptions in \hyperref[lem:2022-06-22, 12:05PM]{Lemma~\ref{lem:2022-06-22, 12:05PM}} are satisfied, and that $b_1, b_2$ go to infinity as $\min(n_1, n_2)\to \infty$. Then, 
	\begin{enumerate}[label=(\roman*)]
		\item under the null hypothesis, as $\min(n_1, n_2)\to \infty$, for all $H\in \NN^+$,
		\begin{align*}
			\left(R_{\bm{n}}^{\bm{m}}, R_{\bm{b}}^{\tilde{\bm{m}}, (I_1)}, \ldots, R_{\bm{b}}^{\tilde{\bm{m}}, (I_H)}\right)&\rightsquigarrow \left(R_{\, \FF}, R_{\, \FF}^{(I_1)}, \ldots, R_{\, \FF}^{(I_H)}\right),\\
			\left(S_{\bm{n}}^{\bm{m}}, S_{\bm{b}}^{\tilde{\bm{m}}, (I_1)}, \ldots, S_{\bm{b}}^{\tilde{\bm{m}}, (I_H)}\right)&\rightsquigarrow \left(S_{\, \FF}, S_{\, \FF}^{(I_1)}, \ldots, S_{\, \FF}^{(I_H)}\right),\\
			\left(T_{\bm{n}}^{\bm{m}}, T_{\bm{b}}^{\tilde{\bm{m}}, (I_1)}, \ldots, T_{\bm{b}}^{\tilde{\bm{m}}, (I_H)}\right)&\rightsquigarrow \left(T_{\, \FF}, T_{\, \FF}^{(I_1)}, \ldots, T_{\, \FF}^{(I_H)}\right),
		\end{align*}
		where $\{R_{\, \FF}^{(I_h)}\}_{h=1}^H$,  $\{S_{\, \FF}^{(I_h)}\}_{h=1}^H$ and $\{T_{\, \FF}^{(I_h)}\}_{h=1}^H$ are independent copies of $R_{\, \FF}$, $S_{\, \FF}$ and $T_{\, \FF}$, respectively.	
		\item under any fixed alternative hypothesis, we have,
		\[
		\lim\limits_{\min(n_1, n_2) \to \infty}\tilde{p}_{\bm{n}, R}= \lim\limits_{\min(n_1, n_2) \to \infty}\tilde{p}_{\bm{n}, S}=\lim\limits_{\min(n_1, n_2) \to \infty}\tilde{p}_{\bm{n}, T}= 0.
		\]
	\end{enumerate}
\end{theorem}
It is important to emphasize that although both subsampling and multiplier replicates converge to the same limit, their differences are discussed in the following remark.
\begin{remark}
	Both multiplier and subsampling replicates converge jointly in distribution to independent copies of the same limiting process, whether under the null hypothesis or a fixed alternative. However, multiplier replicates require additional smoothness conditions on the underlying copula—specifically,~\hyperref[ass:2]{Assumption~\ref{ass:2}}—to ensure convergence. In contrast, subsampling replicates rely on minimal conditions related to the weak convergence of empirical copula processes. This distinction highlights a trade-off between the two methodologies: the multiplier approach offers potentially more precise replicates at the cost of stringent assumptions, whereas the subsampling approach provides greater flexibility and broader applicability.
\end{remark}

\section{Simulation study}
\label{sec:simulation}

\n Finite sample performance is investigated in this section. All tests were conducted with $ 500 $
repetitions under $5\%$ nominal level using $H = 200$ multiplier replicates. Also, given that the statistics
shown in Equation~\eqref{eq:2023-01-24} involve integration, a discrete approximation is applied with a grid spanning a range of $20\times\cdots \times 20$ points on $[0, 1]^d$. To demonstrate the empirical level and power of the equality test, inspired by the methodology of~\cite{Remillard2009}, we generate one sample with a fixed parameter $\theta_0$ for 2 and 3-dimensional Archimedean Clayton copula and elliptical Gaussian copula, respectively. We then introduce a gradual deviation of the parameters from $\theta_0$. It is worth noting that for bivariate copulas, there exists an explicit bijection between Kendall's tau $\tau$ and the corresponding parameter, as outlined in~\cite{Joe1997multivariate}. Therefore, we use $\tau(\theta)$ to denote this relationship. The $\tau(\theta)$ value for $C$ is fixed at $0.2$, while the $\tau(\theta)$ value for $D$ ranges from $0.2$ to $0.7$ in increments of 0.1.

\subsection{Comparison with empirical copula-based test}\label{subsec:2}

\n To show the performance of  empirical Bernstein copula tests, empirical level and power of the proposed
tests were compared with empirical copula-based tests. Let $R_{\bm{n}, mul}$ denote empirical copula-based tests using multiplier bootstrap. Note that $R_{\bm{n}, mul}$ corresponds precisely to the test proposed in~\citet{Remillard2009}, and we implement it using the specialized \textbf{R}~\citep{Rsoftware} package \texttt{TwoCop} by~\citet{TwocopPackage}. Furthermore, let $\{R^{\bm{m}}_{\bm{n}, mul}, S^{\bm{m}}_{\bm{n}, mul}, T^{\bm{m}}_{\bm{n}, mul}\}$ and $\{R^{\bm{m}}_{\bm{n}, sub}, S^{\bm{m}}_{\bm{n}, sub}, T^{\bm{m}}_{\bm{n}, sub}\}$ denote empirical Bernstein copula-based tests using multiplier bootstrap and subsampling techniques, respectively.

The selection of an optimal Bernstein order is important and will be investigated in future work. Possible approaches include stability-based choices over a grid of admissible degrees, resampling-based selection rules, penalized criteria reflecting the bias--variance trade-off, or adaptive procedures inspired by the automatic selection of Legendre coefficients proposed by \citet{NgounouBakam&Pommeret2025}. According to the recommendation of~\citet{Segers2017}, we set $m_r=\lfloor n_r/3 \rfloor$, $r=1,2$, where $\lfloor a\rfloor$ denotes the largest integer less than or equal to $a$. Thus, for the sample sizes considered in the simulation study, this rule gives $m_r=10,16,20,30$ and $66$ when $n_r=30,50,60,90$ and $200$, respectively. For the subsample sizes used in the subsampling method, we follow the recommendation of~\citet{Kojadinovic2019} and take $b_r=\lfloor0.28 n_r \rfloor$, with the associated Bernstein order $m_{(r)}=b_r$ for convenience.

We emphasize that the validity of the multiplier bootstrap is established under the condition
$m_r=c_r n_r^\alpha$ with $3/4<\alpha<1$. Hence, the simulation choice
$m_r=\lfloor n_r/3\rfloor$ should be understood as a practical finite-sample rule and is not directly covered by this sufficient asymptotic condition. In contrast, the subsampling approach is theoretically more flexible, since it allows $\alpha\leq 1$. For this reason, the subsampling-based tests are particularly natural when the Bernstein degree is chosen of linear order in the sample size. We nevertheless also report the multiplier-based tests as finite-sample benchmarks, in order to compare the behavior of the different resampling procedures under the same practical tuning rule.

In Tables~\ref{table:1} and~\ref{table:2}, all the tests effectively maintain their nominal level as the sample size increases. In particular, the proposed Bernstein test is particularly preferable, given that the Cr\'{a}mer-von Mises type statistics consistently exhibit the highest power (indicated by bold font). $R_{\bm{n}, mul}^{\bm{m}}$ and $S_{\bm{n}, mul}^{\bm{m}}$ generally outperform the other statistics. In the case of $3$-dimensional scenarios, Bernstein tests still demonstrate an advantage over empirical copula-based tests, especially for the Gaussian copula models.

When comparing the two resampling methods, it is important to consider their respective trade-offs. The subsampling method exhibits slightly inferior performance compared to the multiplier method but offers the advantage of reduced computational burden. This finding is consistent with the conclusions drawn by~\citet{Kojadinovic2019}. In summary, for small or moderate sample sizes, we recommend utilizing the multiplier method due to its superior performance. However, for large sample sizes, the subsampling method serves as a suitable alternative, striking a balance between computational difficulty and performance.

\begin{table}[h]
	\centering
	\caption{The empirical levels ($\%$) and powers ($\%$) of the tests are estimated based on 500 samples from the 2-dimensional Clayton copula, considering different sample sizes and Kendall's tau for $C$ is fixed at $0.2$, while for $D$ ranges from $0.2$ to $0.8$ in increments of 0.1.}
	\label{table:1}
		\begin{tabular}{lcccccccc}
			\toprule
			Model&$\tau(\theta)$&$R_{\bm{n}, mul}$&$R_{\bm{n}, mul}^{\bm{m}}$ &$R_{\bm{n}, sub}^{\bm{m}}$ &$S_{\bm{n}, mul}^{\bm{m}}$ &$S_{\bm{n}, sub}^{\bm{m}}$&$T_{\bm{n}, mul}^{\bm{m}}$ &$T_{\bm{n}, sub}^{\bm{m}}$\\
			\midrule
			\multicolumn{9}{c}{\small$(n_1, n_2)=(30, 30)$}\\
			\hline
			\multirow{5}{*}{$\rm{Clayton}$}&${0.2}$&$6.2$&$6.6$ &$5.4$&$6.2$&$5.4$&$5.0$&$4.8$ \\
			&$0.3$&$11.6$&$\mathbf{11.8}$ &$10.2$&${11.6}$&$11.2$&$7.6$&$8.4$ \\
			&$0.4$&$19.4$&$\mathbf{20.0}$ &$18.2$&${19.8}$&$18.4$&$16.4$&$18.2$ \\
			&$0.5$&${38.2}$&$\mathbf{40.2}$ &$38.2$&${39.8}$&$37.8$&$31.0$&$35.0$ \\
			&$0.6$&${59.8}$&$\mathbf{60.6}$ &$57.6$&${59.8}$&$59.0$&$48.8$&$54.6$ \\
			&$0.7$&${84.6}$&$\mathbf{84.8}$ &$83.4$&$82.2$&$82.6$ &$67.4$&$77.4$\\
			\midrule
			\multicolumn{9}{c}{\small$(n_1, n_2)=(50, 50)$}\\
			\hline
			\multirow{5}{*}{$\rm{Clayton}$}&${0.2}$&$5.0$&$6.0$ &$3.8$&$4.8$&$3.8$&$4.8$&$4.2$ \\
			&$0.3$&$6.8$&$8.2$ &$7.4$&$\mathbf{8.8}$&$7.2$&$6.0$&$5.6$ \\
			&$0.4$&$27.0$&${27.6}$ &$24.4$&$\mathbf{28.4}$&$24.4$&$21.0$&$20.6$ \\
			&$0.5$&${55.0}$&$57.6$ &$53.4$&$\mathbf{58.2}$&$53.0$&$44.2$&$47.0$ \\
			&$0.6$&${84.6}$&$\mathbf{85.6}$ &$83.6$&${84.8}$&$83.4$&$73.0$&$76.6$ \\
			&$0.7$&${98.0}$&${97.6}$ &$97.0$&$\mathbf{98.0}$&$97.4$ &$94.0$&$95.2$\\
			\botrule
		\end{tabular}
\end{table}

\begin{table}[h]
	\centering
	\caption{The empirical levels ($\%$) and powers ($\%$) of the tests are estimated based on 500 samples from the 2-dimensional Gaussian copula, considering different sample sizes and Kendall's tau for $C$ is fixed at $0.2$, while for $D$ ranges from $0.2$ to $0.8$ in increments of 0.1.}
	\label{table:2}
		\begin{tabular}{lcccccccc}
			\toprule
			Model&$\tau(\theta)$&$R_{\bm{n}, mul}$&$R_{\bm{n}, mul}^{\bm{m}}$ &$R_{\bm{n}, sub}^{\bm{m}}$ &$S_{\bm{n}, mul}^{\bm{m}}$ &$S_{\bm{n}, sub}^{\bm{m}}$&$T_{\bm{n}, mul}^{\bm{m}}$ &$T_{\bm{n}, sub}^{\bm{m}}$ \\
			\midrule
			\multicolumn{9}{c}{\small $(n_1, n_2)=(30, 30)$}\\
			\hline
			\multirow{5}{*}{$\rm{Gaussian}$}&${0.2}$&$6.4$&$6.2$ &$6.0$&$7.2$&$6.0$&$4.2$&$4.2$ \\
			&$0.3$&$9.8$&$10.4$ &$8.6$&$\mathbf{11.2}$&$9.0$&$6.6$&$8.0$ \\
			&$0.4$&$19.0$&$\mathbf{21.2}$ &$18.6$&$20.2$&$17.8$&$15.0$&$16.8$ \\
			&$0.5$&$37.6$&$\mathbf{42.0}$ &$39.0$&$40.6$&$40.2$&$30.8$&$34.0$ \\
			&$0.6$&$61.2$&$\mathbf{63.6}$ &$61.2$&$62.4$&$62.8$&$50.8$&$57.0$ \\
			&$0.7$&$83.2$&${83.8}$ &$83.8$&$82.6$&$\mathbf{84.8}$ &$69.6$&$78.2$\\
			\midrule
			\multicolumn{9}{c}{\small $(n_1, n_2)=(50, 50)$}\\
			\hline
			\multirow{5}{*}{$\rm{Gaussian}$}&${0.2}$&$5.4$&$7.0$ &$5.0$&$6.4$&$5.6$&$4.6$&$4.8$ \\
			&$0.3$&$8.2$&$\mathbf{9.2}$ &$6.6$&${8.4}$&$6.4$&$6.4$&$6.4$ \\
			&$0.4$&$27.2$&${29.4}$ &$25.0$&$\mathbf{31.0}$&$25.2$&$18.2$&$20.0$ \\
			&$0.5$&$61.4$&${62.2}$ &$60.2$&$\mathbf{62.4}$&$59.6$&$48.8$&$51.0$ \\
			&$0.6$&$86.2$&$\mathbf{88.2}$ &$86.6$&$88.0$&$86.6$&$76.0$&$78.0$ \\
			&$0.7$&$98.4$&$98.6$ &$\mathbf{98.8}$&$98.6$&$98.8$ &$93.0$&$95.6$\\
			\botrule
		\end{tabular}
\end{table}

\begin{table}[h]
	\centering
	\caption{The empirical levels ($\%$) and powers ($\%$) of the tests are estimated from 500 samples from $3$-dimensional Clayton and  equicorrated Gaussian copula given various parameters and $50$ sample size. \textbf{Upper Panel:}  parameter $\theta$ for $C$ is fixed at $1.00$, while for $D$ ranges from $1.00$ to $2.50$ in increments of 0.25; \textbf{Lower Panel:} parameter $\theta$ for $C$ is fixed at $0.00$, while for $D$ ranges from $0.00$ to $0.30$ in increments of 0.05.}
	\label{table:3}
		\begin{tabular}{lcccccccc}
			\toprule
			Model&$\theta$&$R_{\bm{n}, mul}$&$R_{\bm{n}, mul}^{\bm{m}}$ &$R_{\bm{n}, sub}^{\bm{m}}$ &$S_{\bm{n}, mul}^{\bm{m}}$ &$S_{\bm{n}, sub}^{\bm{m}}$&$T_{\bm{n}, mul}^{\bm{m}}$ &$T_{\bm{n}, sub}^{\bm{m}}$ \\
			\midrule
			\multirow{5}{*}{$\rm{Clayton}$}&${1.00}$&$3.6$&$2.6$ &$1.8$&$2.8$&$1.6$&$2.0$&$0.4$ \\
			&$1.25$&$5.8$&$6.2$ &$4.6$&$\mathbf{6.8}$&$5.2$&$4.0$&$2.4$ \\
			&$1.50$&$11.4$&$\mathbf{13.4}$ &$10.2$&$12.6$&$9.8$&$9.0$&$5.6$ \\
			&$1.75$&$21.6$&$21.2$ &$18.0$&$\mathbf{22.8}$&$18.2$&$15.2$&$9.6$ \\
			&$2.00$&$\mathbf{28.8}$&$27.8$ &$23.6$&$27.2$&$24.2$&$17.0$&$12.0$ \\
			&$2.25$&$\mathbf{44.6}$&$42.0$ &$38.4$&$42.8$&$40.0$ &$29.6$&$21.0$\\
			&$2.50$&$50.0$&$\mathbf{51.6}$ &$42.0$&$47.6$&$43.2$ &$33.2$&$28.4$\\
			\midrule
			\multirow{7}{*}{$\rm{Gaussian}$}&${0.00}$&$2.6$&$4.0$ &$2.4$&$4.2$&$2.2$&$3.0$&$0.8$ \\
			&$0.05$&$4.4$&$\mathbf{7.4}$ &$3.2$&$6.4$&$4.0$&$3.6$&$1.4$ \\
			&$0.10$&$5.8$&$\mathbf{11.4}$ &$6.6$&$11.4$&$6.0$&$6.8$&$2.2$ \\
			&$0.15$&$13.2$&$\mathbf{18.2}$ &$9.4$&$17.8$&$9.2$&$11.2$&$5.2$ \\
			&$0.20$&$19.6$&$\mathbf{26.2}$ &$18.0$&$25.2$&$17.4$&$15.0$&$7.6$ \\
			&$0.25$&$30.4$&$\mathbf{37.8}$ &$26.0$&$36.0$&$28.6$ &$25.2$&$13.4$\\
			&$0.30$&$47.4$&$\mathbf{52.8}$ &$43.2$&$52.6$&$42.0$ &$37.0$&$22.6$\\
			\botrule
		\end{tabular}

\end{table}
\newpage
\subsection{Comparison with other tests}

\n To explore the performance of the proposed tests, we also compare them with three other updated tests: the empirical beta copula-based test, the class of smooth, possibly data adaptive empirical copulas-based test~\citep{Kojadinovic2024resampling} and the smooth test~\citep{Bakam2024smooth}. Specifically, using the empirical beta copula~\citep[see][]{Segers2017,Kiriliouk2021}, one can construct an equality test with the associated Cram\'{e}r-von Mises test statistic, i.e.,
\begin{align}\label{eq:2024-08-10, 12:57PM}
	R_{\bm{n}}^{\beta}&=n_2{\lambda}_{\bm{n}}\int_{[0, 1]^2}\left\{C_{n_1}^{\beta}(u_1, u_2)- D_{n_2}^{\beta}(u_1, u_2)\right\}^2 \dif u_1\dif u_2,
\end{align}
where, for $j= 1, 2$. Further let $R_{ij}$ and $\widetilde{R}_{ij}$ denote the rank of $X_{ij}$ and  $Y_{ij}$, respectively, 
\begin{align*}
	C_{n_1}^{\beta}(u_1, u_2) & = \frac{1}{n_1}\sum_{i=1}^{n_1}F_{n_1, R_{i1, n_1}}(u_1)F_{n_1, R_{i2, n_1}}(u_2),\\
	D_{n_2}^{\beta}(u_1, u_2) & = \frac{1}{n_2}\sum_{i=1}^{n_2}F_{n_2, \widetilde{R}_{i1, n_2}}(u_1)F_{n_2, \widetilde{R}_{i2, n_2}}(u_2),
\end{align*}
where
\begin{equation*}
	F_{n, r}(u)=\sum_{s=r}^{n}\binom{n}{s}u^s(1-u)^{n-s}
\end{equation*}
is the cumulative distribution function of the beta distribution $\textsf{beta}(r, n+1-r)$. Following~\citet[Algorithm 2.3.2]{Kiriliouk2021}, for both samples, we sample the observations from the empirical beta copulas based on the combined dataset
\begin{equation*}
	\{(X_{11}, \ldots, X_{1d}), \ldots, (X_{n_11}, \ldots, X_{n_1d}), (Y_{11}, \ldots, Y_{1d}), \ldots, (Y_{n_21}, \ldots, Y_{n_2d})\}.
\end{equation*}
Using Equation~\eqref{eq:2024-08-10, 12:57PM}, we obtain the replicates \( \{ R^{\beta, (h)}_{\bm{n}} \}_{h=1}^H \) of \( R_{\bm{n}}^\beta \). Since the replicates \( \{ R^{\beta, (h)}_{\bm{n}} \}_{h=1}^H \) converge to the same limit as \( R_{\bm{n}}^{\bm{m}} \) under the null hypothesis, they are bounded in probability, thereby ensuring a valid empirical level when the null hypothesis is true. Consequently, under any fixed alternative hypothesis and as \( \min(n_1, n_2) \rightarrow \infty \), we have
\[
R_{\bm{n}}^{\beta} \xrightarrow{a.s.} \infty,
\]
which implies that, 
\[
\lim_{\min(n_1, n_2) \to \infty} p^{\beta}_{\bm{n}, R} = 0,
\]
where $p^{\beta}_{\bm{n}, R}= \frac{1}{H}\sum_{h=1}^H \II\left(R^{\beta, (h)}_{\bm{n}} > R_{\bm{n}}^{\beta}\right)$ is the approximate p-value.
Similar results hold for $R_{\bm{n}}^{\text{BetaB4}}$, which is the test statistic based on the estimator in~\cite{Kojadinovic2024resampling}. We implemented this test using the accompanying code from their publication.

The BP24 smooth test, introduced in \citet{Bakam2024smooth}, was implemented using the \texttt{Kcop} package \citep{KcopPackage}. As discussed in preceding sections, the subsampling method is applicable to empirical beta copulas. Therefore, it is included here and denoted as $R_{\bm{n}, sub}^\beta$, along with the proposed Bernstein tests, using the same settings outlined in \hyperref[subsec:2]{Section 5.1}.

\begin{sidewaystable}[htbp]
	\centering
	\caption{The empirical levels ($\%$) and powers ($\%$) of the tests are estimated based on 500 samples. \textbf{Upper Panel:} Bivariate Frank copula with Kendall's tau fixed at $0.2$ for $C$, while for $D$, it ranges from 0.2 to 0.8 in increments of 0.1; \textbf{Lower Panel:} Bivariate Student's t-copula with 3 degrees of freedom and Kendall's $\tau$ fixed at 0.2 for $C$, while $D$ varies across different copula families, all maintaining the same Kendall's $\tau$.}
	\label{table:new}
    \small
		\begin{tabular}{lccccccccccc}
			\toprule
			Model&&BP24&$R_{\bm{n}}^{\beta}$&$R_{\bm{n}, sub}^{\beta}$&$R_{\bm{n}}^{\text{BetaB4}}$&$R_{\bm{n}, mul}^{\bm{m}}$ &$R_{\bm{n}, sub}^{\bm{m}}$ &$S_{\bm{n}, mul}^{\bm{m}}$ &$S_{\bm{n}, sub}^{\bm{m}}$&$T_{\bm{n}, mul}^{\bm{m}}$ &$T_{\bm{n}, sub}^{\bm{m}}$ \\
			\midrule
			&{\small $\tau(\theta)$}&\multicolumn{10}{c}{\small $(n_1, n_2)=(30, 30)$}\\
			\hline
			\multirow{5}{*}{$\rm{Frank}$}&${0.2}$&$12.0$ &5.6&9.6&5.8& 6.8&3.2&5.4&2.8&4.8&2.4\\
			&$0.3$&$14.2$&7.2 &14.6&9.8 &10.2& 5.0&10.8&4.4&7.4&2.4 \\
			&$0.4$&$25.0$&18.2& 28.6& 22.2&22.0&11.8&21.0&11.8&16.6&8.6\\
			&$0.5$&$48.0$&38.8& 53.0&43.4&41.4&26.2&41.6&28.0&31.2&20.6\\
			&$0.6$&$74.8$&60.8& 77.6&70.0&61.4&48.6&64.6&51.0&52.0&38.4\\
			&$0.7$&$93.6$&84.2& 94.2&92.4&80.2&70.8&80.8&71.4&69.2&62.4\\
			\midrule
			&{\small Family}&\multicolumn{10}{c}{\small $(n_1, n_2)=(30, 30)$}\\
			\hline
			\multirow{5}{*}{$\tau=0.2$}&{\small Student}&11.4&6.6 &9.6&7.2&5.8&3.0&6.8&3.0&5.2&2.4\\
			&{\small Frank}&13.6&6.2&10.2&7.2&7.4&3.8&8.0&4.4&5.8&3.0 \\
			&{\small Clayton} &14.2&7.0&11.4&8.6& 8.6&4.6&7.8&4.4&7.0&3.0\\
			&{\small Gaussian}&13.6&6.0&10.6&6.6&7.4&4.0&8.0&3.6&5.4&2.6\\
			&{\small Gumbel}&14.6&6.6&10.8&7.6&8.0&4.4&7.8&3.4&5.8&2.6\\
			\botrule
		\end{tabular}
\end{sidewaystable}
\begin{table}[h]
	\centering
	\caption{The empirical levels ($\%$) of the tests are estimated based on 500 samples with increasing sample sizes.  \textbf{Upper Panel:}  Settings corresponding to the Upper Panel in~\hyperref[table:new]{Table~\ref{table:new}}; \textbf{Lower Panel:} Settings corresponding to the Lower Panel in~\hyperref[table:new]{Table~\ref{table:new}}.}
	\label{table:new2}
		\begin{tabular}{lccccccccc}
			\toprule
			\multirow{1}{*}{\small Model}&\multirow{1}{*}{\small $(n_1, n_2)$}&\multicolumn{8}{c}{\small Test}\\
			\hline
			&&{\footnotesize  BP24}&{\footnotesize$R_{\bm{n}, sub}^{{\tiny \beta}}$} &{\footnotesize $R_{\bm{n}, mul}^{\bm{m}}$}&{\footnotesize  $R_{\bm{n}, sub}^{\bm{m}}$}&{\footnotesize  $S_{\bm{n}, mul}^{\bm{m}}$}&{\footnotesize  $S_{\bm{n}, sub}^{\bm{m}}$} &{\footnotesize  $T_{\bm{n}, mul}^{\bm{m}}$}&{\footnotesize  $T_{\bm{n}, sub}^{\bm{m}}$ }\\
			\midrule
			\multirow{5}{*}{upper}&{\small (30, 30)}&12.0&9.6&6.8&3.2 &5.4 &2.8&4.8&2.4\\
			&{\small (60, 30)}&9.0&9.8&6.0&4.0 &7.0 &3.2&5.0&2.6\\
			&{\small (60, 60)}&9.6&8.6&5.8&4.0 & 5.6&3.8&4.0&2.2\\
			&{\small (90, 60)}&8.2&8.2&5.4&4.0 & 4.8&3.8&4.6&3.0\\
			&{\small (90, 90)}&8.6&6.2&4.0&2.2 & 3.8&2.0&2.6&1.4\\
			&{\small (200, 200)}&6.0&6.4&4.6&4.0 & 4.4&4.0&5.6&3.2\\
			\midrule  
			\multirow{5}{*}{lower}&{\small (30, 30)}&11.4&11.2&5.8&3.0 &6.8 &3.9&5.2&2.4\\
			&{\small (60, 30)}&9.6&9.6&6.4&4.4 &6.2 &4.4&4.6&3.2\\
			&{\small (60, 60)}&9.6&10.2&5.6&3.8 & 5.4&3.0&4.2&3.0\\
			&{\small (90, 60)}&7.4&8.6&5.2&3.8 & 5.0&3.6&4.6&3.0\\
			&{\small (90, 90)}&8.0&6.8&3.8&2.6 & 3.8&2.2&2.4&1.0\\
			&{\small (200, 200)}&6.0&5.2&5.0&3.4&4.6 &4.0&4.8&3.0\\
			\botrule
		\end{tabular}
\end{table}

From~\hyperref[table:new]{Table~\ref{table:new}}, it can be observed that the $R_{\bm{n}, sub}^\beta$ is the most powerful  test when the copula family is fixed and Kendall's tau varies, while BP24 is the most powerful test when Kendall's tau is fixed and the copula family varies.  The $R_{\bm{n}}^\beta$ and $R_{\bm{n}}^{\text{\tiny BetaB4}}$  tests, which are based on bootstrap methods, show fewer advantages compared to the Bernstein tests with multiplier bootstrap.

While the BP24 and $R_{\bm{n}, sub}^\beta$ tests exhibit higher power, this advantage comes at the expense of an inflated Type I error rate. In contrast, the proposed Bernstein tests maintain an empirical level closely aligned with the nominal level, ensuring statistical validity and reliability. This adherence to the nominal level is crucial for minimizing false positives, thereby providing more trustworthy results. To further investigate this phenomenon,~\hyperref[table:new2]{Table~\ref{table:new2}} illustrates how the empirical level changes with increasing sample size. Notably, the BP24 test fails to achieve a reasonable empirical level until the sample size reaches (200, 200),  whereas the $R_{\bm{n}, sub}^\beta$ test achieves an acceptable  empirical level with a sample size of (90, 90).
\section{Insurance data application}
\label{sec: application}
\n 	Consider the same insurance data analyzed in~\cite{Remillard2009} from the Society of Actuaries Group Medical Insurance Large Claims Database, which includes claims information for each claimant categorized by status (employee or dependent) in 1991. The extracted data contain the losses (hospital charges) and ALAEs (other charges) for dependent females (967 observations) compared to employee females (1116 observations) aged 30-39.

\begin{sidewaystable}[htbp]
	\centering
	\caption{The p-values of the tests for the insurance data evaluate the null hypothesis of equal dependence structure between losses and ALAEs for dependent females compared to employee females.}
	\label{table:4}
		\begin{tabular}{lcccccccccccc}
			\toprule
			Test&$R_{\bm{n}, mul}$&$R_{\bm{n}}^{\beta}$&$R_{\bm{n}, sub}^{\beta}$&BP24&$R_{\bm{n}}^{\text{BetaB4}}$&$R_{\bm{n}, mul}$&$R_{\bm{n}, mul}^{\bm{m}}$ &$R_{\bm{n}, sub}^{\bm{m}}$ &$S_{\bm{n}, mul}^{\bm{m}}$ &$S_{\bm{n}, sub}^{\bm{m}}$&$T_{\bm{n}, mul}^{\bm{m}}$ &$T_{\bm{n}, sub}^{\bm{m}}$ \\
			\midrule
			P-value&$0.065$&0.085&0.095&0.219&0.085&$0.065$&$\mathbf{0.030}$&$0.125$&$0.070$&$0.160$&$0.085$&$0.260$ \\
			Time ratio&---&---&---&---&---&---&100\%&10.4\%&100\%&10.3\%&100\%&10.4\% \\
			\botrule
		\end{tabular}
\end{sidewaystable}

All the empirical Bernstein copula-based tests were conducted in the same environment as in the simulation study, except that we used $m_r=\lfloor n_r/40 \rfloor$, $r=1,2$, for the Bernstein degrees. Since the sample sizes are $n_1=967$ and $n_2=1116$, this gives $m_1=24$ and $m_2=27$, respectively. This choice keeps the Bernstein degrees in a moderate range, approximately between 20 and 30, and avoids the substantially larger degrees that would result from applying the simulation rule $m_r=\lfloor n_r/3\rfloor$ to this larger data set.

The results of this application should be interpreted with some caution. The numerical values of the test statistics are not directly comparable, since the statistics are defined on different scales and are combined with different resampling schemes. In terms of $p$-values, only the empirical Bernstein copula-based test using the multiplier bootstrap rejects the null hypothesis at the 5\% level, whereas the other tests do not. Therefore, the application provides suggestive, but not uniform, evidence of a difference between the two dependence structures. This illustrates that, even for relatively large samples, different copula equality tests may have different sensitivities to particular features of the alternative. It also supports the usefulness of reporting several complementary procedures. Furthermore, the computation time ratio indicates that the subsampling method is approximately ten times faster than the multiplier bootstrap method.

\section{Concluding remarks}
\n New tests for the equality of multivariate copulas have been proposed, using Bernstein polynomials, and their properties have been thoroughly examined. These tests are based on the empirical tests introduced by~\citet{Remillard2009} and demonstrate better performance in simulation studies. The limiting distribution of the proposed test statistics has been thoroughly investigated, and resampling methods using Bernstein polynomials have been developed and implemented to accurately simulate p-values.

However, the performance of these tests deteriorates as the dimensionality increases. It is essential to explore methods to mitigate the impact of dimensionality in nonparametric copula tests. Moreover, it would be intriguing to extend our results to $k$-sample equality tests of copulas~\citep[as seen in][]{Bouzebda&Keziou&Zari2011,Quessy2016general,Quessy2021szekely,Derumigny2022,Bakam2024smooth} and to equality tests of copulas involving a large number of populations~\citep[as explored in][]{Zhan2014,Cousido2019,Jimenez2022}. 

Such extensions, however, would require careful attention to the curse of dimensionality. Indeed, for tensor-product Bernstein approximations, the number of Bernstein coefficients increases rapidly with the copula dimension, for instance as $(m_r+1)^d$ when the same degree $m_r$ is used in each coordinate. This issue becomes even more pronounced in $k$-sample settings or when a large number of populations is considered, since several empirical copulas and resampling procedures must be handled simultaneously. It is particularly relevant for multiplier-bootstrap implementations, where the computational burden may become substantial.

Possible ways to mitigate this difficulty include projection-based or random-projection procedures
\citep[][]{Cuesta2007random}, pairwise or decomposition-based approaches inspired by pair-copula constructions
\citep[][]{Aas2009pair}, and sparse-grid or sparse-approximation methods
\citep[][]{Bungartz2004sparse}. A detailed theoretical and numerical investigation of these possibilities is beyond the scope of the present paper and is left for future research.

\bmhead{Acknowledgements}

M. Belalia gratefully acknowledges the research support
of the Natural Sciences and Engineering Research Council of Canada (RGPIN/05496-2020).

\begin{appendices}

\section{Proofs of the main results\label{app}}

\vspace*{12pt}
\begin{lemma}\label{lem:2022-06-23, 3:36AM}
	Let $C$  be a copula function with continuous margins, then
	\begin{equation*}
		\sup_{\bm{u}\in [0, 1]^d}\left|{C}_{n, m}(\bm{u})-C(\bm{u})\right|\overset{a.s.}{=}O\left(\sqrt{\frac{\log\log n}{n}}\right)+O\left(m^{-1/2}\right)
	\end{equation*}
	as $n\to \infty$.
\end{lemma}

\begin{proof}

	From~\citet{Deheuvels1981}  and reference therein, a multivariate version of Glivenko-Cantelli theorem holds uniformly, $i.e.$,
	\begin{equation*}
		\sup_{\bm{u}\in [0, 1]^d}\left| C_n(\bm{u})-C(\bm{u})\right|\overset{a.s.}{=}O\left(\sqrt{\frac{\log\log n}{n}}\right).
	\end{equation*}
	The following steps are adapted from~\citet[Proof of Theorem 1]{Jansen2012} for the multivariate case. Let $G^B_{n, m}$ denote the Bernstein copula, that is, the margins are known. Then, by triangle inequality,
	\begin{equation*}
		\sup_{\bm{u}\in [0, 1]^d}\left|C_{n, m}(\bm{u})-C(\bm{u})\right|\le \sup_{\bm{u}\in [0, 1]^d}\left|C_{n, m}(\bm{u})-G^B_{n, m}(\bm{u})\right|+\sup_{\bm{u}\in [0, 1]^d}\left|G^B_{n, m}(\bm{u})-C(\bm{u})\right|.
	\end{equation*}
	Further, let $\bm{k}=(k_1, \ldots, k_d)$,
	\begin{align*}
		\sup_{\bm{u}\in [0, 1]^d}\left|C_{n, m}(\bm{u})-G^B_{n, m}(\bm{u})\right|&=\sup_{\bm{u}\in [0, 1]^d}\left|\sum_{k_1=0}^m\cdots \sum_{k_d=0}^m\left[C_n\left(\bm{k}/m\right)-C(\bm{k}/m)\right]\prod_{\ell=1}^dP_{k_{\ell}, m}(u_{\ell})\right|\\
		&\le \max_{0\le k_1,\ldots, k_d\le m}\left|C_n(\bm{k}/m)-C(\bm{k}/m)\right|\\
		&\le \sup_{\bm{u}\in [0, 1]^d}\left|C_n(\bm{u})-C(\bm{u})\right|=O\left(\sqrt{\frac{\log\log n}{n}}\right),
	\end{align*}
	where the first inequality follows by the fact that $\sum_{k_1=0}^m\cdots \sum_{k_d=0}^m\prod_{\ell=1}^dP_{k_{\ell}, m}(u_{\ell})=1$ for all $u_\ell\in [0, 1]$. And using Lipschitz property of copulas, 
	\begin{align*}
		\sup_{\bm{u}\in [0, 1]^d}\left|G^B_{n, m}(\bm{u})-C(\bm{u})\right|&=\sup_{\bm{u}\in [0, 1]^d}\left|\sum_{k_1=0}^m\cdots \sum_{k_d=0}^m\left[C(\bm{k}/m)-C(\bm{u})\right]\prod_{\ell=1}^dP_{k_{\ell}, m}(u_{\ell})\right|\\
		&\le\sup_{\bm{u}\in [0, 1]^d}\left|\sum_{k_1=0}^m\cdots \sum_{k_d=0}^m\left|k_1/m-u_1\right| \prod_{\ell=1}^dP_{k_{\ell}, m}(u_{\ell})+\cdots\right.\\
		&\quad +\left. \sum_{k_1=0}^m\cdots \sum_{k_d=0}^m \left|k_d/m-u_d\right| \prod_{\ell=1}^dP_{k_{\ell}, m}(u_{\ell})\right|.
	\end{align*}
	Let $\{B_{{\ell}}: \ell= 1, \ldots, d\}$ denote binomial random variables with parameters $m$ and $u_{\ell}$. Then, by the Cauchy-Schwarz inequality,
	\begin{align*}
		\sup_{\bm{u}\in [0, 1]^d}\left|G^B_{n, m}(\bm{u})-C(\bm{u})\right|&\le \sum_{\ell=1}^d \sup_{u_{\ell}\in [0, 1]} \left| \sqrt{m^{-2}\Var \left(B_{{\ell}}\right)}\right|\\
		&=\sum_{\ell=1}^d \sup_{u_{\ell}\in [0, 1]} \sqrt{m^{-1}u_{\ell}(1-u_{\ell})}=\frac{d}{2} \,m^{-1/2}.
	\end{align*}
	This completes the proof.
\end{proof}

\subsection{Proof of~\hyperref[lem:2022-06-22, 12:05PM]{Lemma~\ref{lem:2022-06-22, 12:05PM}}}

Since the two samples are independent, clearly, 
\begin{equation*}
	\left(\sqrt{1-\lambda_{\bm{n}} }\CC_{n_1, m_1}, \sqrt{\lambda_{\bm{n}} }\DD_{n_2, m_2}\right)\rightsquigarrow\left( \sqrt{1-\lambda}\CC, \sqrt{\lambda}\DD\right), \quad \text{ in } \left(\ell^{\infty}\left([0,1]^d\right)\right)^{\otimes 2}.
\end{equation*}
Hence, a direct application of the continuous mapping theorem leads to the desired result.\qed 

\subsection{Proof of~\hyperref[thm:2022-06-22, 12:31PM]{Theorem~\ref{thm:2022-06-22, 12:31PM}}}

The weak convergence of $R_{\bm{n}}^{\bm{m}}, T_{\bm{n}}^{\bm{m}}$ follows directly from the continuous mapping theorem and \hyperref[lem:2022-06-22, 12:05PM]{Lemma~\ref{lem:2022-06-22, 12:05PM}}. For $S_{\bm{n}}^{\bm{m}}$, note that
\begin{equation*}
	S_{\bm{n}}^{\bm{m}}=\int_{[0, 1]^d}\Big\{\FF^{\bm{m}}_{\bm{n}}(\bm{u})\Big\}^2\dif \mathscr{C}_{\bm{n}}^{\bm{m}}(\bm u),
\end{equation*}
where $\{\FF_{\bm{n}}^{\bm{m}}\}$ is a sequence of asymptotically tight processes with respect to  the uniform norm on the $\mathscr{C}\left([0, 1]^d\right)$ (the space of continuous functions defined on $[0, 1]^d$). Then, by applying~\citet[Proposition 4]{Genest2013regularized} and the fact that 
\begin{equation}
	\int_{[0, 1]^d} k(\bm{u}) \dif \mathscr{C}_{\bm{n}}^{\bm{m}}(\bm u) \xrightarrow[]{a.s.}\int_{[0, 1]^d} k(\bm{u}) \dif C(\bm{u}),
\end{equation}
where $k\in \mathscr{C}\left([0, 1]^d\right)$, one has 
\begin{equation}\label{eq:2024-11-13, 10:11AM}
	\Delta_n = \int_{[0, 1]^d}\big\{\FF^{\bm{m}}_{\bm{n}}(\bm{u})\big\}^2 \dif  \mathscr{C}_{\bm{n}}^{\bm{m}}(\bm u)- \int_{[0, 1]^d}\big\{\FF^{\bm{m}}_{\bm{n}}(\bm{u})\big\}^2 \dif {C}(\bm{u}) \xrightarrow[]{p}0.
\end{equation}
Further, by continuous mapping theorem, one has
\begin{equation}\label{eq:2024-11-13, 10:12AM}
	\int_{[0, 1]^d}\big\{\FF^{\bm{m}}_{\bm{n}}(\bm{u})\big\}^2 \dif {C}(\bm{u}) \rightsquigarrow \int_{[0, 1]^d}\big\{\FF^{}_{}(\bm{u})\big\}^2 \dif {C}(\bm{u}).
\end{equation}
Combing Equations~\eqref{eq:2024-11-13, 10:11AM} and~\eqref{eq:2024-11-13, 10:12AM} completes the proof. \qed

\subsection{Proof of \hyperref[thm:2022-06-23, 3:40AM]{Theorem~\ref{thm:2022-06-23, 3:40AM}}}

The strongly uniform consistency of the $d$-variate empirical Bernstein copula is provided in \hyperref[lem:2022-06-23, 3:36AM]{Lemma~\ref{lem:2022-06-23, 3:36AM}}. Then, by the continuous mapping theorem and independence between the two samples, it follows immediately that $n_2^{-1}R_{\bm{n}}^{\bm{m}}$ and $\sqrt{n_2^{-1}}T_{\bm{n}}^{\bm{m}}$ converge to $R(C, D)$ and $T(C, D)$ almost surely, respectively. Further, to prove the convergence of $n_2^{-1} S_{\bm{n}}^{\bm{m}}$, write
\begin{align*}
	\left|n_2^{-1} S_{\bm{n}}^{\bm{m}}-S(C, D)\right| &= \left| \lambda_{\bm{n}}\int_{[0, 1]^d}\left\{{C}_{n_1, m_1}(\bm{u})-{D}_{n_2, m_2}(\bm{u})\right\}^2\dif \mathscr{C}_{\bm{n}}^{\bm{m}}(\bm u) \right.\\
    & \quad \left. - \lambda\int_{[0, 1]^d}\Big\{C(\bm{u})-D(\bm{u})\Big\}^2\dif \left(\lambda C(\bm{u})+(1-\lambda) D(\bm{u})\right) \right|\\
	& \le  \lambda_{\bm{n}} |\gamma_{\bm{n}}^{\bm{m}}|+|\Gamma_{\bm{n}}^{\bm{m}}|,
\end{align*}
where 
\begin{align*}
	\gamma_{\bm{n}}^{\bm{m}}&=\int_{[0, 1]^d}\left\{{C}_{n_1, m_1}(\bm{u})-{D}_{n_2, m_2}(\bm{u})\right\}^2\dif \mathscr{C}_{\bm{n}}^{\bm{m}}(\bm u)
	-\int_{[0, 1]^d}\left\{{C}(\bm{u})-{D}(\bm{u})\right\}^2\dif \mathscr{C}_{\bm{n}}^{\bm{m}}(\bm u),
\end{align*}
and 
\begin{align*}
	\Gamma_{\bm{n}}^{\bm{m}}&= \lambda_{\bm{n}}\int_{[0, 1]^d}\left\{{C}(\bm{u})-{D}(\bm{u})\right\}^2\dif \mathscr{C}_{\bm{n}}^{\bm{m}}(\bm u)
	- \lambda\int_{[0, 1]^d}\left\{{C}(\bm{u})-{D}(\bm{u})\right\}^2\dif \left(\lambda C(\bm{u})+(1-\lambda) D(\bm{u})\right).
\end{align*}
Since 
\begin{align*}
	&\left|\{{C}_{n_1, m_1}(\bm{u})-{D}_{n_2, m_2}(\bm{u})\}^2-\left\{{C}(\bm{u})-{D}(\bm{u})\right\}^2\right|\\
	&\quad =\left|\left[{C}_{n_1, m_1}(\bm{u})+C(\bm{u})\right]-\left[{D}_{n_2, m_2}(\bm{u})+D(\bm{u})\right]\right|\\
	&\qquad \times\left|\left[{C}_{n_1, m_1}(\bm{u})-C(\bm{u})\right]-\left[{D}_{n_2, m_2}(\bm{u})-D(\bm{u})\right]\right|\\
	&\quad\le \left[\left|{C}_{n_1, m_1}(\bm{u})+C(\bm{u})\right|+\left|{D}_{n_2, m_2}(\bm{u})+D(\bm{u})\right|\right]\\
	&\qquad \times\left[\left|{C}_{n_1, m_1}(\bm{u})-C(\bm{u})\right|+\left|{D}_{n_2, m_2}(\bm{u})-D(\bm{u})\right|\right]\\
	&\quad\le  4\sup_{\bm{u}\in [0, 1]^d}\left|{C}_{n_1, m_1}(\bm{u})-C(\bm{u})\right|\\
	&\qquad +4\sup_{\bm{u}\in [0, 1]^d}\left|{D}_{n_2, m_2}(\bm{u})-D(\bm{u})\right|,
\end{align*}
using \hyperref[lem:2022-06-23, 3:36AM]{Lemma~\ref{lem:2022-06-23, 3:36AM}}, 
\begin{align}\label{eq:gammanm}
	|\gamma_{\bm{n}}^{\bm{m}}|&\le 4\left\{\sup_{\bm{u}\in [0, 1]^d}\left|{C}_{n_1, m_1}(\bm{u})-C(\bm{u})\right|+\sup_{\bm{u}\in [0, 1]^d}\left|{D}_{n_2, m_2}(\bm{u})-D(\bm{u})\right|\right\}
	\xrightarrow[\min(n_1, n_2)\to \infty]{a.s.} 0.
\end{align}
Now, \hyperref[lem:2022-06-23, 3:36AM]{Lemma~\ref{lem:2022-06-23, 3:36AM}} ensures that $ C_{n_1, m_1}$ and $D_{n_2, m_2}$ converge uniformly almost sure to $C$  and $D$, receptively. Hence, $\mathscr{C}_{\bm{n}}^{\bm{m}}$ converge almost surely to $\lambda C + (1-\lambda) D$.  In addition the function $h: \bm{u} \in [0,1]^d \mapsto \{C(\bm{u})-D(\bm{u})\}^2 $ is bounded and continuous $[0,1]^d$. Then, using the Portmanteau theorem, one can conclude that 
\begin{align*}
	\int_{[0, 1]^d}\left\{{C}(\bm{u})-{D}(\bm{u})\right\}^2\dif \mathscr{C}_{\bm{n}}^{\bm{m}}(\bm u)
	\xrightarrow[\min(n_1, n_2)\to \infty]{a.s.} \int_{[0, 1]^d}\left\{{C}(\bm{u})-{D}(\bm{u})\right\}^2\dif \left(\lambda C(\bm{u})+(1-\lambda) D(\bm{u})\right),
\end{align*}
which ensure the convergence almost surely of $\Gamma_{\bm{n}}^{\bm{m}}$ to $0$. 
Combining the convergence of $\gamma_{\bm{n}}^{\bm{m}}$ and $\Gamma_{\bm{n}}^{\bm{m}}$, allow us to conclude that $n_2^{-1}S^{\bm{m}}_{\bm{n}}$ converges to $S(C, D)$ almost surely. \qed

\subsection{Proof of \hyperref[lem:2023-09-16, 7:59AM]{Lemma~\ref{lem:2023-09-16, 7:59AM}}}\label{pf:lem}

Under these assumptions, one needs to use the framework in~\citet{Segers2017}. Specifically, let $\mu_{{m_1}, \bm{u}}$ be the law of random vector $(B_1/m_1, \ldots, B_d/m_1)$, where $\{B_\ell: \ell=1, \ldots, d\}$ follows $\textsf{Binomial}(m_1, u_\ell)$. The empirical Bernstein copula in Equation~\eqref{eq:2022-07-10, 11:01AM} can be rewritten as
\begin{equation*}
	{C}_{n_1, m_1}(\bm{u})=\int_{[0, 1]^d}{C}_{n_1}(\bm{w})\,\dif\mu_{m, \bm{u}}(\bm{w}), \qquad \bm{u}\in [0 ,1]^d.
\end{equation*}
Moreover, write $\bm{w}(t)=\bm{u}+t(\bm{w}-\bm{u})$ with $t\in [0, 1]$. Then, the empirical Bernstein copula process is
\begin{align}\label{eq:2023-08-24, 6:52PM}
	\CC_{n_1, m_1}(\bm{u})&=\sqrt{n_1}\left\{	{C}_{n_1, m_1}(\bm{u})-C(\bm{u})\right\}\notag\\
	&=\sqrt{n_1}\left\{	{C}_{n_1, m_1}(\bm{u})-\int_{[0, 1]^d}C(\bm{w})\,\dif\mu_{m_1, \bm{u}}(\bm{w})\right.\notag\\
	&\quad \left.+\int_{[0, 1]^d}C(\bm{w})\,\dif\mu_{m_1, \bm{u}}(\bm{w})-C(\bm{u})\right\}\notag\\
	&=\int_{[0, 1]^d}\sqrt{n_1}\left\{	{C}_{n_1}(\bm{w})-C(\bm{w})\right\}\,\dif \mu_{m_1, \bm{u}}(\bm{w})\notag\\
	&\quad +\sqrt{n_1}\left\{\int_{[0, 1]^d}C(\bm{w})\,\dif\mu_{m_1, \bm{u}}(\bm{w})-C(\bm{u})\right\}\notag\\
	&=T_1+T_2.
\end{align}
The two terms are dealt with separately.
\begin{itemize}
	\item For the term $T_1$, according to~\citet[Proposition 3.1]{Segers2017},  one has 
	\begin{align*}
		&\sup\limits_{\bm{u}\in [0, 1]^d}\left|\int_{[0, 1]^d}\sqrt{n_1}\left\{{C}_{n_1}(\bm{w})-C(\bm{w})\right\}\,\dif\mu_{m_1, \bm{u}}(\bm{w})-\sqrt{n_1}\left\{	{C}_{n_1}(\bm{u})-C(\bm{u})\right\}\right|=o_p(1).
	\end{align*}
	And note that, $\sqrt{n_1}\left\{	{C}_{n_1}-C\right\}\rightsquigarrow\BB_C$ in $\ell^{\infty}\left([0, 1]^d\right)$ under~\hyperref[ass:2022-07-05, 1:23PM]{Assumption~\ref{ass:2022-07-05, 1:23PM}} (see~\citet{Segers2012}). Therefore, $T_1\rightsquigarrow \BB_C$ as $n$ goes to infinity.
	
	\item For the term $T_2$, Let $m_1=cn_1^{\alpha}$ for some $c >0$, by~\citet[Lemma 3.1]{Kojadinovic2022stute}, under~\hyperref[ass:2022-07-05, 1:23PM]{ Assumption~\ref{ass:2022-07-05, 1:23PM}-\ref{ass:2} }, one has 
	\begin{align*}
		&\sup_{\bm{u}\in [0, 1]^d}\sqrt{n_1}\left|\int_{[0, 1]^d}C(\bm{w})\,\dif\mu_{m_1, \bm{u}}(\bm{w})-C(\bm{u})\right|=O\left(n_1^{(3-4\alpha)/6}\right)
	\end{align*}
	almost surely. Therefore,  if $\alpha > 3/4$, $T_2$ goes to zero as $n$ goes to infinity.

\end{itemize}
Combining the above results completes the proof.	\qed

\subsection{Proof of \hyperref[thm:2022-06-23, 3:40AM]{Theorem~\ref{thm:2023-02-05, 8:49PM}}}

One can verify an intermediate result that as $\min(n_1, n_2)\to \infty$
\begin{equation}\label{eq:2023-02-09, 8:52AM}
	\sup_{\bm{u}\in [0, 1]^d}\left|\FF_{\bm{n}}^{\bm{m}}(\bm{u})-\FF_{\bm{n}}^*(\bm{u})\right|\overset{a.s.}{=}O\left(\Psi(n_1, n_2)\right).
\end{equation}
Indeed, let $\FF_{\bm{n}}:=\sqrt{1-\lambda_{\bm{n}} }\CC_{n_1}-\sqrt{\lambda_{\bm{n}} }\DD_{n_2}$, in the light of~\citet[Proof of Theorem 2.1]{Bouzebda&ElFaouzi&Zari2011}. By~\citet{Borisov1982}, for a sequence of Brownian bridge $\{\BB_{n_1, C}\}_{n_1\in \NN^+}$, which are independent copies of $\BB_C$, one has as $n_1\to \infty$,
\begin{equation*}
	\sup_{\bm{u}\in [0, 1]^d}\left|\GG_{n_1}(\bm{u})-\BB_{n_1, C}(\bm{u})\right|\overset{a.s.}{=}O\left(n_1^{-1/(2(2d-1))} \log n_1\right),
\end{equation*}
where 
\begin{equation*}
	\GG_{n_1}(\bm{u})=\sqrt{n}\left\{G_{n_1}(\bm{u})-C(\bm{u})\right\}=\frac{1}{\sqrt{n_1}}\sum_{i=1}^{n_1}\left[\prod_{\ell=1}^d\II\left(U_{i\ell}\le u_{\ell}\right)-C(\bm{u})\right].
\end{equation*}
Then, under our assumptions, by~\citet{Segers2012}, as $n_1\to \infty$, 
\begin{equation*}
	\sup_{\bm{u}\in [0, 1]^d}\left|\CC_{n_1}(\bm{u})-\GG_{n_1}(\bm{u})+\sum_{\ell=1}^d\GG_{n_1}(\bm{u}^{\ell})\dot{C}_{\ell}(\bm{u})\right|\overset{a.s.}{=}O\left(n_1^{-1/4}(\log n_1)^{1/2}(\log \log n_1)^{1/4}\right).
\end{equation*}
Since $\dot{C}_{\ell}(\bm{u}), \ell=1, \ldots, d$ are bounded, 
\begin{align*}
	\sup_{\bm{u}\in [0, 1]^d}\left|\CC_{n_1}(\bm{u})-\CC^{(n_1)}(\bm{u})\right|&\le \sup_{\bm{u}\in [0, 1]^d}\left|\CC_{n_1}(\bm{u})-\GG_{n_1}(\bm{u})+\sum_{\ell=1}^d\GG_{n_1}(\bm{u}^{\ell})\dot{C}_{\ell}(\bm{u})\right|\\
	&\quad +\sup_{\bm{u}\in [0, 1]^d}\left|\GG_{n_1}(\bm{u})-\sum_{\ell=1}^d\GG_{n_1}(\bm{u}^{\ell})\dot{C}_{\ell}(\bm{u})-\CC^{(n_1)}(\bm{u})\right|\\
	&=O\left(n_1^{-1/4}(\log n_1)^{1/2}(\log \log n_1)^{1/4}\right) +O\left(n_1^{-1/(2(2d-1))} \log n_1\right)\\
	&=O\left(n_1^{-1/(2(2d-1))} \log n_1\right)
\end{align*} 
almost surely as $n_1\to \infty$. Similarly, one can obtain
\begin{equation*}
	\sup_{\bm{u}\in [0, 1]^d}\left|\DD_{n_2}(\bm{u})-\DD^{(n_2)}(\bm{u})\right|=O\left(n_2^{-1/(2(2d-1))} \log n_2\right)
\end{equation*}
almost surely as $n_2\to \infty$. Therefore, one has almost surely
\begin{equation*}
	\sup_{\bm{u}\in [0, 1]^d}\left|\FF_{\bm{n}}^{}(\bm{u})-\FF_{\bm{n}}^*(\bm{u})\right|=O\left(\Psi(n_1, n_2)\right).
\end{equation*}
Therefore,
\begin{align}\label{eq:2023-02-09, 7:32AM}
	\sup_{\bm{u}\in [0, 1]^d}\left|\FF_{\bm{n}}^{\bm{m}}(\bm{u})-\FF_{\bm{n}}^*(\bm{u})\right|&\le	\sup_{\bm{u}\in [0, 1]^d}\left|\FF_{\bm{n}}^{\bm{m}}(\bm{u})-\FF_{\bm{n}}(\bm{u})\right|+	\sup_{\bm{u}\in [0, 1]^d}\left|\FF_{\bm{n}}^{}(\bm{u})-\FF_{\bm{n}}^*(\bm{u})\right| \notag \\
	&\le \sqrt{n_2 \lambda_{\bm{n}}} \left[\sup_{\bm{u}\in [0, 1]^d}\left|C_{n_1, m_1}(\bm{u})-C_{n_1}(\bm{u})\right|\right.\notag\\
	&\quad \left.+ \sup_{\bm{u}\in [0, 1]^d}\left|D_{n_2, m_2}(\bm{u})-D_{n_2}(\bm{u})\right|\right]+O\left(\Psi(n_1, n_2)\right)\notag\\
	&= \sqrt{1-\lambda_{\bm{n}}} \sup_{\bm{u}\in [0, 1]^d}\left|\CC_{n_1, m_1}(\bm{u})-\CC_{n_1}(\bm{u})\right|\notag\\
	&\quad+ \sqrt{\lambda_{\bm{n}}} \sup_{\bm{u}\in [0, 1]^d}\left|\DD_{n_2, m_2}(\bm{u})-\DD_{n_2}(\bm{u})\right| +O\left(\Psi(n_1, n_2)\right).
\end{align}
Further, adapted from~\citet{Segers2017}, under our assumptions, as $\min(n_1, n_2)\to \infty$, 
\begin{align}\label{eq:2023-02-14, 10:51AM}
	\sup_{\bm{u}\in [0, 1]^d}\left|\CC_{n_1, m_1}(\bm{u})-\CC_{n_1}(\bm{u})\right|&=O\left(n_1^{-1/2}\right),\notag\\
	\sup_{\bm{u}\in [0, 1]^d}\left|\DD_{n_2, m_2}(\bm{u})-\DD_{n_2}(\bm{u})\right|&=O\left(n_2^{-1/2}\right).
\end{align}
One has 
\begin{align*}
	\eqref{eq:2023-02-09, 7:32AM}&=O\left(n_1^{-1/2}\right)+O\left(n_2^{-1/2}\right)+O\left(\Psi(n_1, n_2)\right)=O\left(\Psi(n_1, n_2)\right).
\end{align*}
Based on the intermediate result Equation~\eqref{eq:2023-02-09, 8:52AM}, one can show strong approximation of these statistics.
\begin{itemize}
	\item  For $R_{\bm{n}}^{\bm{m}}$, 
	\begin{align*}
		\left|R_{\bm{n}}^{\bm{m}}-\int_{[0, 1]^d}\{\FF^*_{\bm{n}}(\bm{u})\}^2\dif \bm{u}\right|&= \left|\int_{[0, 1]^d}\{\FF^{\bm{m}}_{\bm{n}}(\bm{u})\}^2-\{\FF^*_{\bm{n}}(\bm{u})\}^2\dif \bm{u}\right|\\
		&\le \left(\sup_{\bm{u}\in [0, 1]^d}\left|\FF_{\bm{n}}^{\bm{m}}(\bm{u})-\FF_{\bm{n}}^{*}(\bm{u})\right|\right)\\
		&\quad \cdot \left(\sup_{\bm{u}\in [0, 1]^d}\left|\FF_{\bm{n}}^{\bm{m}}(\bm{u})+\FF_{\bm{n}}^{*}(\bm{u})\right|\right)\\
		&\le \left(\sup_{\bm{u}\in [0, 1]^d}\left|\FF_{\bm{n}}^{\bm{m}}(\bm{u})-\FF_{\bm{n}}^{*}(\bm{u})\right|\right)\\
		&\quad \cdot \left(\sup_{\bm{u}\in [0, 1]^d}\left|\FF_{\bm{n}}^{\bm{m}}(\bm{u})-\FF_{\bm{n}}^{*}(\bm{u})\right|+2\sup_{\bm{u}\in [0, 1]^d}\left|\FF_{\bm{n}}^{\bm{m}}(\bm{u})\right|\right)\\
		&=O\left(\max\left((\log\log n_1)^{1/2}, (\log \log n_2)^{1/2}\right)\Psi(n_1, n_2)\right),
	\end{align*}
	where the last equality follows by~\citet[Theorem 1]{Jansen2012}.
	
	\item For $S_{\bm{n}}^{\bm{m}}$, 
	\begin{align*}
		&\left|S_{\bm{n}}^{\bm{m}}-\int_{[0, 1]^d}\{\FF^*_{\bm{n}}(\bm{u})\}^2\dif C(\bm{u})\right|\\
		&\quad\le 	\left|S_{\bm{n}}^{\bm{m}}-\int_{[0, 1]^d}\{\FF^*_{\bm{n}}(\bm{u})\}^2\dif \mathscr{C}_{\bm{n}}^{\bm{m}}(\bm u)\right|\\
		&\qquad + 	\left|\int_{[0, 1]^d}\{\FF^*_{\bm{n}}(\bm{u})\}^2\dif \mathscr{C}_{\bm{n}}^{\bm{m}}(\bm u)-\int_{[0, 1]^d}\{\FF^*_{\bm{n}}(\bm{u})\}^2\dif C(\bm{u})\right|\\
		&\quad=O\left(\max\left((\log\log n_1)^{1/2}, (\log \log n_2)^{1/2}\right)\Psi(n_1, n_2)\right)\\
		&\qquad + 	\left|\int_{[0, 1]^d}\{\FF^*_{\bm{n}}(\bm{u})\}^2\dif\, \left\{ \mathscr{C}_{\bm{n}}^{\bm{m}}(\bm u)-C(\bm{u})\right\}\right|\\
		&\quad\le O\left(\max\left((\log\log n_1)^{1/2}, (\log \log n_2)^{1/2}\right)\Psi(n_1, n_2)\right)\\
		&\qquad + \sup_{\bm{u}\in [0, 1]^d}\left|\{\FF^*_{\bm{n}}(\bm{u})\}^2\right| \sup_{\bm{u}\in [0, 1]^d} \left|\mathscr{C}_{\bm{n}}^{\bm{m}}(\bm u)-C(\bm{u})\right|\\
		&\quad=O\left(\max\left((\log\log n_1)^{1/2}, (\log \log n_2)^{1/2}\right)\Psi(n_1, n_2)\right)\\
		&\qquad + O(1) O\left((n_1+n_2)^{-1/2}(\log \log n_1+n_2)^{1/2}\right)\\
		&\quad=O\left(\max\left((\log\log n_1)^{1/2}, (\log \log n_2)^{1/2}\right)\Psi(n_1, n_2)\right),
	\end{align*}
	where the $O(1)$ comes form the fact that $\FF_{\bm{n}}^*$ is a centred Gaussian process.
	
	\item For $T_{\bm{n}}^{\bm{m}}$, the result immediately follows.
\end{itemize}
This completes the proof. 	\qed

\subsection{Proof of~\hyperref[thm:2022-07-08, 10:37PM]{Theorem~\ref{thm:2022-07-08, 10:37PM}}}

Part $(i)$ can be adapted using the approach outlined in the proof of Theorem 1. Here, we present the proof for Part $(ii)$. Under assumed conditions, the centered Gaussian processes $\CC$ and $\DD$, as defined in Equations~\eqref{eq:2023-09-16, 11:13AM} and~\eqref{eq:2024-10-07, 11:17AM}, respectively, have continuous trajectories and are bounded; see~\cite{Segers2012}. Consequently, the process $\FF$, being a linear combination of $\CC$ and $\DD$, also exhibits continuity and boundedness. Moreover, for each $h \in \{1, \ldots, H\}$, the process $\FF_{\bm{n}}^{\bm{m}, (h)}$ converges weakly to $\FF$ on the space $\ell^{\infty}\left([0, 1]^d\right)$, that is,

\begin{equation*}
	\FF_{\bm{n}}^{\bm{m}, (h)} \rightsquigarrow \FF.
\end{equation*}
This convergence implies that $\sup_{\bm{u} \in [0, 1]^d} \left|\FF_{\bm{n}}^{\bm{m}, (h)}(\bm{u})\right|$ is bounded in probability.

Note that, as presented in~\hyperref[thm:2022-06-23, 3:40AM]{Theorem~\ref{thm:2022-06-23, 3:40AM}}, one has, as $\min(n_1, n_2)\to \infty$. 
\begin{equation*}
	R_{\bm{n}}^{\bm{m}}\xrightarrow{a.s.}\infty,  \quad 	S_{\bm{n}}^{\bm{m}}\xrightarrow{a.s.}\infty ,\quad 	T_{\bm{n}}^{\bm{m}}\xrightarrow{a.s.}\infty. 
\end{equation*}
To establish that the approximate p-values approach zero as the sample size increases, we need to demonstrate that the multiplier replicates $\{R_{\bm{n}}^{\bm{m}, (h)}\}_{h=1}^H$, $\{S_{\bm{n}}^{\bm{m}, (h)}\}_{h=1}^H$, and $\{T_{\bm{n}}^{\bm{m}, (h)}\}_{h=1}^H$ are bounded in probability. We address each of these replicates separately as follows. For each $h \in \{1, \ldots, H\}$,

\begin{itemize}
	\item For $R_{\bm{n}}^{\bm{m}, (h)}$, we have
	\begin{align*}
		R_{\bm{n}}^{\bm{m}, (h)} &= \int_{[0, 1]^d} \left\{\FF^{\bm{m}, (h)}_{\bm{n}}(\bm{u})\right\}^2 \, \dif\bm{u},
	\end{align*}
	which can be bounded as
	\begin{align*}
		R_{\bm{n}}^{\bm{m}, (h)} &\leq \left\{\sup_{\bm{u} \in [0, 1]^d} \left|\FF_{\bm{n}}^{\bm{m}, (h)}(\bm{u})\right|\right\}^2 \cdot 1.
	\end{align*}

	Thus $R_{\bm{n}}^{\bm{m}, (h)}$ is bounded in probability;
	
	\item For $S_{\bm{n}}^{\bm{m}, (h)}$, we have
	\begin{align*}
		S_{\bm{n}}^{\bm{m}, (h)} &= \int_{[0, 1]^d} \left\{\FF^{\bm{m}, (h)}_{\bm{n}}(\bm{u})\right\}^2 \, \dif \mathscr{C}_{\bm{n}}^{\bm{m}}(\bm u).
	\end{align*}
	Let $V_{C, D}$ be the measure induced by $\mathscr{C}_{\bm{n}}^{\bm{m}}$, which is defined as, for every rectangle $[\bm{u}, \bm{v}] \in [0, 1]^d$, 
	\begin{equation*}
		V_{C, D}([\bm{u}, \bm{v}])=\sum_{\bm{c}}(-1)^{s(\bm{c})}\mathscr{C}_{\bm{n}}^{\bm{m}}(\bm{c}),
	\end{equation*}
	where the summation is taken over all the vertices $\bm{c}$ of $[\bm{u}, \bm{v}]$ ($i.e.$, each $c_\ell$ is either $u_\ell$ or $v_\ell$) and $s(\bm{c})$ is the number of indices $\ell$'s such that $c_\ell=u_\ell$.
	Then, we have
	\begin{align*}
		S_{\bm{n}}^{\bm{m}, (h)} &\leq \left\{ \sup_{\bm{u} \in [0, 1]^d} \left| \FF_{\bm{n}}^{\bm{m}, (h)}(\bm{u}) \right| \right\}^2 \cdot V_{C, D}([0, 1]^d) \\
		&= \left\{ \sup_{\bm{u} \in [0, 1]^d} \left| \FF_{\bm{n}}^{\bm{m}, (h)}(\bm{u}) \right| \right\}^2 \cdot 1,
	\end{align*}
	which implies that $S_{\bm{n}}^{\bm{m}, (h)}$ is bounded in probability.

	\item For $T_{\bm{n}}^{\bm{m}, (h)}$, we have
	\begin{align*}
		T_{\bm{n}}^{\bm{m}, (h)} &= \sup_{\bm{u} \in [0, 1]^d} \left|\FF^{\bm{m}, (h)}_{\bm{n}}(\bm{u})\right|,
	\end{align*}
	which is bounded in probability.
\end{itemize}
This concludes the proof. \qed

\subsection{Proof of \hyperref[prop:2023-02-10, 7:47AM]{Proposition~\ref{prop:2023-02-10, 7:47AM}}}

We only prove the results for a sample $X$ with size $n_1$. According to~\citet{Remillard2009}, for $\bm{u}\in [0, 1]^d$ and $h\in \{ 1,\ldots, H\}$, in the space $\ell^{\infty}\left([0, 1]^d\right)$, $\GG_{n_1}=\sqrt{n}\left\{G_{n_1}-C\right\}\rightsquigarrow \BB_C$ and $\GG_{n_1}^{(h)}=\rightsquigarrow \BB_C$, where

\begin{equation*}
	\GG_{n_1}(\bm{u})=\frac{1}{\sqrt{n_1}}\sum_{i=1}^{n_1}\left[\prod_{\ell=1}^d\II\left(U_{i\ell}\le u_{\ell}\right)-C(\bm{u})\right],
\end{equation*}
and 
\begin{equation*}
	\GG_{n_1}^{(h)}(\bm{u})=\frac{1}{\sqrt{n_1}}\sum_{i=1}^{n_1}\left(\xi_i^{(h)}-\overline{\xi}^{(h)}_{n_1}\right)\left[\prod_{\ell=1}^{d}\II\left(\widehat{U}_{i\ell}\le u_{\ell}\right)\right].	
\end{equation*}
By~\hyperref[lem:2023-09-16, 7:59AM]{Lemma~\ref{lem:2023-09-16, 7:59AM}},  
one has 
\begin{equation*}
	\CC_{n_1,m_1}=\sqrt{n_1}\{C_{n_1, m_1}-C\}\rightsquigarrow\CC,
\end{equation*}
where $      \CC(\bm{u})  =\BB_C(\bm{u})-\sum_{\ell=1}^{d}\BB_C\left(\bm{u}^{\ell}\right)\dot{C}_{\ell}(\bm{u})$, which implies that
\begin{equation*}
	\GG_{n_1, m_1}\rightsquigarrow \BB_C.
\end{equation*}
Furthermore, given that $\CC_{n_1,m_1}$ and $\CC_{n_1}$ converge to the same limit, one has that $\GG_{n_1, m_1}^{(h)}$ and $\GG_{n_1}^{(h)}$ will also converge to the same limit. This leads to the conclusion that $\GG_{n_1, m_1}^{(h)}\rightsquigarrow \BB_C$. \qed

\subsection{Proof of \hyperref[prop:2023-02-10, 7:51AM]{Proposition~\ref{prop:2023-02-10, 7:51AM}}}

We only show the uniform consistency for 
\begin{align*}
	\partial C_{n_1, m_1}(\bm{u})/\partial u_1&=m_1\sum_{k_1=0}^{m_1-1}\cdots \sum_{k_{d}=0}^{m_1}\left\{C_{n_1}\left(\frac{k_1+1}{m_1},\ldots, \frac{k_d}{m_1}\right)\right. -\left. C_{n_1}\left(\frac{k_1}{m_1},\ldots,  \frac{k_d}{m_1}\right)\right\}\\
	&\quad \times P_{k_1, m_1-1}(u_1)\cdots P_{k_d, m_1}(u_d).
\end{align*}
The result for other partial derivatives can be obtained similarly. For any $\{\bm{u}\in [0, 1]^d: u_1 \in [b_{n_1}, 1-b_{n_1}]\}$, one has 
\begin{align*}
	&	\left|\partial C_{n_1, m_1}(\bm{u})/\partial u_1-\dot{C}_1(\bm{u})\right|\\
	&\quad \le \left|m_1\sum_{k_1=0}^{m_1-1}\cdots \sum_{k_{d}=0}^{m_1}\left\{C_{n_1}\left(\frac{k_1+1}{m_1},\ldots, \frac{k_d}{m_1}\right)\right. -\left. C_{n_1}\left(\frac{k_1}{m_1},\ldots,  \frac{k_d}{m_1}\right)\right.\right.\\
	&\qquad \left.\left.-C_{}\left(\frac{k_1+1}{m_1},\ldots, \frac{k_d}{m_1}\right)+ C_{}\left(\frac{k_1}{m_1},\ldots,  \frac{k_d}{m_1}\right)\right\} P_{k_1, m_1-1}(u_1)\cdots P_{k_d, m_1}(u_d)\right|\\
	&\qquad +\left|m_1\sum_{k_1=0}^{m_1-1}\cdots \sum_{k_{d}=0}^{m_1}\left\{C_{}\left(\frac{k_1+1}{m_1},\ldots, \frac{k_d}{m_1}\right)- C_{}\left(\frac{k_1}{m_1},\ldots,  \frac{k_d}{m_1}\right)\right\}\right.\\
	&\qquad \left.\cdot P_{k_1, m_1-1}(u_1)\cdots P_{k_d, m_1}(u_d)-\dot{C}_1(\bm{u})\right|\\
	&\quad =A_1+A_2.
\end{align*}
Further, let $P'_{k_1, m_1}(u_1)$ be the derivative of $P_{k_1, m_1}(u_1)$, then
\begin{align*}
	A_1&\le \sum_{k_1=0}^{m_1}\cdots \sum_{k_{d}=0}^{m_1}\left|C_{n_1}\left(\frac{k_1}{m_1},\ldots, \frac{k_d}{m_1}\right)\right. -\left. C_{}\left(\frac{k_1}{m_1},\ldots,  \frac{k_d}{m_1}\right)\right| \left|P'_{k_1, m_1}(u_1)\right|\cdots P_{k_d, m_1}(u_d)\\
	&\le \sup_{\bm{u}\in [0, 1]^d} \left|C_{n_1}(\bm{u})-C(\bm{u})\right|\cdot \sum_{k_1=0}^{m_1}\left|P'_{k_1, m_1}(u_1)\right|\\
	&=O\left(m_1^{1/2}n_1^{-1/2}(\log\log n_1)^{1/2}\right)
\end{align*}
almost surely as $n_1\to \infty$ and where $\sum_{k_1=0}^{m_1}\left|P'_{k_1, m_1}(u_1)\right|=O\left(m_1^{1/2}\right)$ by~\citet[Lemma 1]{Jansen2014}.

For dealing with $A_2$, let $\nu_{m_1, \bm{u}}$ be the law of random vector $(S_1/(m_1-1), B_2/m_1, \ldots, B_d/m_1)$, where $S_1$ follows \textsf{Binomial}$(m_1-1, u_1)$ and $\{B_\ell: \ell=2, \ldots, d\}$ follow $\textsf{Binomial}(m_1, u_\ell)$. One has 
\begin{align*}
	&\sum_{k_1=0}^{m_1-1}\cdots \sum_{k_{d}=0}^{m_1}\left\{C_{}\left(\frac{k_1+1}{m_1},\ldots, \frac{k_d}{m_1}\right)- C_{}\left(\frac{k_1}{m_1},\ldots,  \frac{k_d}{m_1}\right)\right\}P_{k_1, m_1-1}(u_1)\cdots P_{k_d, m_1}(u_d)\\
	&\quad =\int_{[0, 1]^d}C\left(\left(w_1+\frac{1}{m_1-1}\right)\frac{m_1-1}{m_1}, w_2, \dots, w_d\right)\dif \nu_{m_1, \bm{u}}(\bm{w})\\
	&\qquad -\int_{[0, 1]^d}C\left(w_1\frac{m_1-1}{m_1}, w_2, \dots, w_d\right)\dif \nu_{m_1, \bm{u}}(\bm{w}).
\end{align*}
Using the representation in~\citet[Proof of Proposition 3.4]{Segers2017}, for $0<t<1$, one has 
\begin{align}\label{eq:2023-09-17, 12:06PM}
	A_2&=\left|m_1\int_{[0, 1]^d}\left[C\left(\left(w_1+\frac{1}{m_1-1}\right)\frac{m_1-1}{m_1}, w_2, \dots, w_d\right)\right.\right.\notag\\
	&\quad \left. \left.-C\left(w_1\frac{m_1-1}{m_1}, w_2, \dots, w_d\right)\right]\dif \nu_{m_1, \bm{u}}(\bm{w})-\dot{C}_1(\bm{u})\right|\notag\\
	&=\left|\int_0^1\int_{[0, 1]^d}\left[\dot{C}_1\left(\left(\frac{m_1-1}{m_1}w_1+\frac{1+t}{m_1}\right), w_2, \dots, w_d\right)-\dot{C}_1(\bm{u})\right]\dif \nu_{m_1, \bm{u}}(\bm{w})\dif t\right|.
\end{align}
Let $w_1'(w_1, t):=\frac{m_1-1}{m_1}w_1+\frac{1+t}{m_1}$ and $\varepsilon_{n_1}=b_{n_1}/2$, then 
\begin{align*}
	\eqref{eq:2023-09-17, 12:06PM} &\le \left|\int_0^1\int_{[0, 1]^d}\left[\dot{C}_1\left(w_1', w_2, \dots, w_d\right)-\dot{C}_1(\bm{u})\right]\right.\\
	&\quad \cdot 	\left.\II\left(\max(|w_1'-u_1|, |w_2-u_2|,  \dots, |w_d-u_d|)\le \varepsilon_{n_1}\right)\dif \nu_{m_1, \bm{u}}(\bm{w})\dif t\right|  \\
	& \quad +\left|\int_0^1\int_{[0, 1]^d}\left[\dot{C}_1\left(w_1', w_2, \dots, w_d\right)-\dot{C}_1(\bm{u})\right]\right.\\
	&\quad \cdot \left.\II\left(\max(|w_1'-u_1|, |w_2-u_2|,  \dots, |w_d-u_d|)> \varepsilon_{n_1}\right)\dif \nu_{m_1, \bm{u}}(\bm{w})\dif t\right|\\
	&=A_{21}+A_{22}.
\end{align*}
Further, the two terms are dealt with separately using the strategy in~\citet[Proof of Lemma 3.1]{Kojadinovic2022stute},
\begin{itemize}
	\item For term $A_{21}$, under \hyperref[ass:2022-07-05, 1:23PM]{Assumption~\ref{ass:2022-07-05, 1:23PM}-\ref{ass:2}}, by~\citet[Lemma 4.3]{Segers2012}, for a constant $L>0$, one has
	\begin{align*}
		&\left|\dot{C}_1\left(w_1', w_2, \ldots, w_d\right)-\dot{C}(\bm{u})\right|\II\left(\max(|w_1'-u_1|, |w_2-u_2|,  \dots, |w_d-u_d|)\le \varepsilon_{n_1}\right)\\
		&\quad \le Lb_{n_1}^{-1}\left[|w_1'-u_1|+|w_2-u_2|+\cdots+|w_d-u_d|\right].
	\end{align*}
	Further, 
	\begin{align}\label{eq:2023-09-17, 7:49PM}
		A_{21} & \le  L b_{n_1}^{-1}\int_{[0, 1]^d} \int_0^1 \left[|w_1-u_1|+|w_2-u_2|+\cdots+|w_d-u_d|\right.\notag\\
		&\quad \left.+|w_1'-w_1|\right] \dif t \dif \nu_{m_1, \bm{u}}(\bm{w})\notag\\
		&\le L b_{n_1}^{-1}\int_{[0, 1]^d} \int_0^1 \left[|w_1-u_1|+|w_2-u_2|+\cdots+|w_d-u_d|\right.\notag\\
		&\quad \left.+\left| \frac{1+t}{m_1}\right|+\left|\frac{w_1}{m_1}\right|\right] \dif t \dif \nu_{m_1, \bm{u}}(\bm{w})\notag\\
		& =L b_{n_1}^{-1}\int_{[0, 1]^d}  \left[|w_1-u_1|+|w_2-u_2|+\cdots+|w_d-u_d|\right.\notag\\
		&\quad \left.+ \frac{3}{2m_1}+\frac{w_1}{m_1}\right]  \dif \nu_{m_1, \bm{u}}(\bm{w})
	\end{align}
	By Cauchy-Schwarz inequality,
	\begin{align*}
		\eqref{eq:2023-09-17, 7:49PM} & \le Lb_{n_1}^{-1}\left[O\left(m_1^{-1/2}\right)+O\left(m_1^{-1}\right)\right] =O\left(b_{n_1}^{-1}m_1^{-1/2}\right).
	\end{align*}	
	
	\item For term $A_{22}$, since $\dot{C}_1\in [0, 1]$ and using the result in the proof of $A_{21}$, 
\end{itemize}
\begin{align}\label{eq:2023-09-17, 7:52PM}
	A_{22}&\le \left|\int_0^1\int_{[0, 1]^d}\frac{2}{\varepsilon_{n_1}}\max(|w_1'-u_1|, |w_2-u_2|, \ldots,|w_d-u_d|) \dif \nu_{m_1, \bm{u}}(\bm{w}) \dif t\right|\notag\\
	&\le \left|\int_0^1\int_{[0, 1]^d}\frac{2}{\varepsilon_{n_1}}\max(  |w_1'-u_1|+|w_2-u_2|+\cdots+|w_d-u_d|) \dif \nu_{m_1, \bm{u}}(\bm{w}) \dif t\right|\notag\\
	&\le  \varepsilon_{n_1}^{-1}\left[O\left(m_1^{-1/2}\right)+O(m_1^{-1})\right]=O\left(b_{n_1}^{-1}m_1^{-1/2}\right).
\end{align}
Therefore, 
\begin{align*}
	\sup_{\{\bm{u}\in [0, 1]^d: u_1\in [b_{n_1}, 1-b_{n_1}]\}}\left|\frac{\partial C_{n_1, m_1}(\bm{u}) }{\partial u_1}-\dot{C}_1(\bm{u})\right|&=O\left(m_1^{1/2}n_1^{-1/2}(\log \log n_1)^{1/2}\right)
	+O\left(b_{n_1}^{-1}m_1^{-1/2}\right)
\end{align*}
almost surely as $n_1\to \infty$, which completes the proof. \qed

\subsection{Proof of \hyperref[thm:2022-07-10, 10:27PM]{Theorem~\ref{thm:2022-07-10, 10:27PM}}} 

We show only the result for $\CC$. The $\DD$ process can be treated in a similar way. The empirical copula for subsample $h_1$ is defined as 
\begin{equation*}
	C^{(h_1)}_{b_1}(\bm{u})=\frac{1}{b_{1, h_1}}\sum_{i=1}^{b_{1, h_1}}\prod_{\ell=1}^{d}\II\left(\widehat{U}_{i\ell}\le u_{\ell}\right),\qquad h_1\in \{h_{1,1}, \ldots, h_{1, H}\}.
\end{equation*}
The corresponding corrected subsample empirical copula process is
\begin{equation*}
	\CC_{b_1}^{(h_1)}=\sqrt{\frac{b_1}{1-b_1/n_1}}\{C^{(h_1)}_{b_1}-C_{n_1}\}.
\end{equation*}
By~\citet[Theorem 1]{Kojadinovic2019}, under our assumptions, in the space $\ell^{\infty}\left([0, 1]^d\right)^{\otimes (H+1)}$,
\begin{equation*}
	\left(\CC_{n_1}, \CC_{b_1}^{(h_{1, 1})},\ldots, \CC_{b_1}^{(h_{1, H})}\right)\rightsquigarrow \left(\CC, \CC^{(1)}, \ldots, \CC^{(H)}\right).
\end{equation*}
Then, one can show that $\CC_{b_1, m_{(1)}}^{(h_1)}$ and $\CC_{b_1}^{(h_1)}$ are asymptotically equivalent. Indeed, applying~\citet[Theorem 3.6]{Segers2017}, as $n_1\to \infty$,
\begin{align*}
	\sqrt{\frac{b_1}{1-b_1/n_1}}\{C^{(h_1)}_{b_1, m_{(1)}}-C\}&=\sqrt{\frac{b_1}{1-b_1/n_1}}\{C^{(h_1)}_{b_1}-C_{}\}+o_p(1),\\
	\sqrt{\frac{b_1}{1-b_1/n_1}}\{C^{}_{n_1, m_{1}}-C\}&=\sqrt{\frac{b_1}{1-b_1/n_1}}\{C^{}_{n_1}-C_{}\}+o_p(1).
\end{align*}
Therefore, $\CC_{b_1, m_{(1)}}^{(h_1)}=\CC_{b_1}^{(h_1)}+o_p(1), \quad n_1\to \infty,$ which completes the proof. \qed 

\end{appendices}

\bibliography{sn-bibliography}%

\end{document}